\documentclass[a4paper]{article}
\usepackage[english]{babel}
\usepackage{latexsym}
\usepackage{amsmath,amssymb,amsfonts,amsthm,graphicx}
\usepackage{color}
\usepackage[normalem]{ulem}

\newtheorem{thm}{Theorem}
\newtheorem{pro}[thm]{Proposition}
\newtheorem{lem}[thm]{Lemma}
\newtheorem{cor}[thm]{Corollary}
\theoremstyle{definition}

\theoremstyle{remark}
\newtheorem{rem}[thm]{{\bf Remark}}
\newtheorem{example}{{\bf Example}}

\newcommand{\af}{\mathrm{afrk}}
\newcommand{\x}{{\bf x}}
\newcommand{\y}{{\bf y}}
\newcommand{\bb}{{\bf b}}
\newcommand{\ba}{{\bf a}}
\newcommand{\al}{\alpha}
\newcommand{\D}{{\cal D}}
\newcommand{\ov}{\overline}
\newcommand{\found}{\mathrm{found}}
\newcommand{\vol}{\mathrm{vol}}

\newcommand{\RM}{\mathrm{RM}}
\newcommand{\aut}{\mathrm{Aut}}
\renewcommand{\thefootnote}

\title{On the volumes and affine types of trades}
\author{ E. Ghorbani$^{\,\rm 1,2,}$\thanks{Corresponding author},~  S. Kamali$^{\,\rm 2}$,~ G.B. Khosrovshahi$^{\,\rm 2}$,~  D.S. Krotov$^{\,\rm 3}$\\[.4cm]
{\normalsize\sl $^{\rm 1}$Department of Mathematics, K.N. Toosi University of Technology,}\\
{\normalsize\sl P. O. Box 16765-3381, Tehran, Iran}\\
{\normalsize\sl $^{\rm 2}$School of Mathematics, Institute for Research in Fundamental Sciences (IPM),}\\
{\normalsize\sl P. O. Box 19395-5746, Tehran, Iran }\\
{\normalsize\sl $^{\rm 3}$Sobolev Institute of Mathematics, Novosibirsk, 630090, Russia}}
 
\begin{document}
\maketitle
\footnotetext{{\em E-mail Addresses}:  {\tt e\_ghorbani@ipm.ir} (E. Ghorbani), {\tt sakamali@ipm.ir} (S. Kamali), {\tt rezagbk@ipm.ir} (G.B. Khosrovshahi),
  {\tt krotov@math.nsc.ru} (D.S. Krotov)}

 \begin{abstract}
 A $[t]$-trade is a pair $T=(T_+, T_-)$ of disjoint collections
of subsets (blocks) of a $v$-set $V$ such that for every $0\le i\le t$,
any $i$-subset of $V$ is included in the same number
of blocks of $T_+$ and of $T_-$. It follows that $|T_+| = |T_-|$ and this common value is called the volume of $T$.
If we restrict all the blocks to have the same size, we obtain the  classical $t$-trades as a special case of
$[t]$-trades.
It is known that the minimum volume of a nonempty $[t]$-trade is $2^t$.
Simple $[t]$-trades (i.e., those with no repeated blocks) correspond to a Boolean function of degree at most $v-t-1$.
From the characterization of Kasami--Tokura of such functions with small number of ones,
it is known that  any simple $[t]$-trade of volume at most $2\cdot2^t$ belongs to one of two
affine types, called Type\,(A) and Type\,(B) where Type\,(A) $[t]$-trades are known to exist. By considering the affine rank, we prove that
$[t]$-trades of Type\,(B) do not exist.
  Further, we derive the spectrum of volumes of simple trades up to $2.5\cdot 2^t$,
  extending the known result for volumes less than $2\cdot 2^t$.
  We also give a characterization of ``small" $[t]$-trades for $t=1,2$. Finally, an algorithm to produce $[t]$-trades for specified  $t$, $v$
is given. The result of the implementation of the algorithm for $t\le4$, $v\le7$ is reported.

\noindent {\bf AMS Mathematics Subject Classification\,(2010):} 05B05, 05B15, 97B05
\end{abstract}

\section{Introduction}

Let $v$, $k$, $t$ be  positive integers such that $v> k> t$ and $V$ be a $v$-set. Suppose that $T_+$ and $T_-$ are two disjoint collections of $k$-subsets of $V$ (called {\em blocks})  such that
the occurrences of every $t$-subset of $V$ in $T_+$ and $T_-$ are the same. Then $T=(T_+, T_-)$
is called a $t$-$(v, k)$ {\em trade} (or a $t$-{\em trade} when the role of $v$, $k$ is not important).
Basically, $t$-trades have been defined and utilized in connection
with $t$-designs: if $\D_1$ and $\D_2$ are two $t$-designs with the same parameters and the same ground set $V$, then
$(\D_1\setminus\D_2,\, \D_2\setminus\D_1)$ is a $t$-trade.
In this paper we consider $[t]$-{\em trades}, a
 generalization of  $t$-trades,
relaxed in the sense that the block size is not fixed. More precisely, a $[t]$-trade is a pair $T=(T_+, T_-)$ of disjoint collections
of subsets of $V$ such that for every $0\le i\le t$,
every $i$-subset of $V$ is included in the same number
of blocks of $T_+$ and of $T_-$.
Note that any $t$-trade is also an $i$-trade for every $0\le i\le t$, which means that any $t$-trade is a $[t]$-trade as well.
On the other hand, $[t]$-trades can be naturally treated as trades of orthogonal arrays:
given two orthogonal binary arrays $A_1$, $A_2$ with the same parameters and strength $t$, their difference pair
$(A_1\backslash A_2,A_2\backslash A_1)$ is a $[t]$-trade (here, each array is treated as the set of its row-tuples).

For a $[t]$-trade $T=(T_+, T_-)$ we have $|T_+| = |T_-|$ and this common value is called the {\em volume} of $T$ and denoted by $\vol(T)$.
It is known that the smallest volume of a nonempty $t$-trade is $2^t$ which was determined independently in
\cite{Hwang:PhD,Hwang:86} and \cite{FranklPach:83}. For the volumes (of $t$-trades) between $2^t$ and $2\cdot2^t$, it was conjectured by Khosrovshahi and Malik \cite{Khos:1990:trades,Malik:88}  and by Mahmoodian and Soltankhah \cite{Soltankhah:88,MahSol:1992}   (see also \cite{HedKho:trades}) that any volume in this range is of the form
$2^{t+1}-2^i$ for some $i\in\{0,\ldots, t-1\}$. This was known as ``the gaps conjecture'' which was proved recently in \cite{Kro:gaps} for simple trades
(for the trades with repeated blocked, the problem remains open).
 We note that  the spectrum of volumes of $t$-trades and that of $[t]$-trades are the same \cite{Kro:gaps} (i.e., a $t$-trade of volume $m$ exists if and only if a $[t]$-trade of volume $m$ exists).
 This is a key observation which allows one to translate  problems related to the volumes of $t$-trades to the setting of $[t]$-trades;  the strategy which was employed in settling the gaps conjecture for simple trades \cite{Kro:gaps}.
  Further important problems in design theory can be described in terms of volumes of trades.
  For instance, the celebrated halving conjecture \cite{Hartman:halving} can be considered as a partial case of the problem of determining the maximum volume of  $t$-$(v,k)$ trades (which is conjectured to be $\frac{1}{2}\binom{v}{k}$ whenever $\binom{v-i}{k-i}$
is even for all $i=0,\ldots,t$). This is one of the motivations to study $[t]$-trades as a new tool to attack problems in combinatorial design theory which can be  described in terms of (volumes) of $t$-trades.

In this paper we further study $[t]$-trades and their volumes.
As noted in \cite{Kro:gaps}, any simple (i.e., with no repeated blocks) $[t]$-trade corresponds to a Boolean function of degree at most $v-t-1$ (where $v$ is the number of arguments).
From the characterization  of such functions with small number of ones (given in \cite{KasamiTokura:70}),
it is observed that  any simple $[t]$-trade of volume at most $2\cdot2^t$ belongs to one of the two
affine types, called Type\,(A) and Type\,(B) (Type\,(A) $[t]$-trades are known to exist).
Existence of  $[t]$-trades of Type\,(B) was  declared as an open problem in \cite{Kro:gaps}.
 By considering the affine rank, we prove that $[t]$-trades of Type\,(B) do not exist.
Also from our results on affine rank of trades, we derive the spectrum of volumes of trades up to $2.5\cdot 2^t$ extending the gaps conjecture proved in \cite{Kro:gaps}. 

The paper is organized as follows.
Section~\ref{s:def} contains main definitions. In Section~\ref{s:prel} we prove some auxiliary statements.
In Section~\ref{s:rank}, we consider the affine rank of 
simple $[t]$-trades. We utilize these considerations to prove the non-existence of simple $[t]$-trades of Type\,(B) as well as simple $[t]$-trades of volume $2^{t+1}+2^i$, $(t-1)/2\le i\le t-4$. Based on this latter non-existential result
and the construction of $[t]$-trades of volumes $2^{t+1}+2^{t-1}-2^i$, $0\le i\le t-2$, and
$2^{t+1}+2^{t-1}-3\cdot 2^i$, $0\le i\le t-3$, in  Section~\ref{s:spectr} we characterize
the spectrum of volumes of simple $[t]$-trades up to the value $2.5\cdot 2^{t}$ exclusively. Section~\ref{s:vol3,6} is devoted to the characterization
of $[1]$-trades of volume $3$ and $[2]$-trades of volume $6$.
Section~\ref{s:comp} contains the results of an exhaustive computer enumeration of the equivalence classes of trades for small $t$, and small foundations and volumes.

Finally, we note that our results are applicable to the classical  $t$-trades.
Indeed, on one hand, the $t$-trades are a special case of the $[t]$-trades;
on the other hand, every $[t]$-trade can be mapped to a $t$-trade with a fixed block size by some affine transformation \cite{Kro:gaps}.
However, the characterization results for $[t]$-trades do not imply that
the corresponding $t$-trades are also characterized up to isomorphism.
Indeed, the class of equivalence transformations for $[t]$-trades is larger than that of $t$-trades (it contains shifts), and nonisomorphic $t$-trades could be equivalent as $[t]$-trades.
As an example of the characterization of small $t$-trades, we mention
the classification in \cite[Table~3.4]{FGG:2004:trades} of the Steiner $2$-trades with block size $3$, volume at most $9$ and foundation size at most $11$,
where the additional ``Steiner'' property means that
no pair of elements is included in more than one block of each leg of the trade.

\section{Definitions}\label{s:def}
\subsection{$[t]$-trades}\label{s:[t]-trades}

Let $t$, $v$ be positive integers with $t<v$.
The subsets of $V = \{ 1, \ldots , v\}$
will be associated with their characteristic $v$-tuples, e.g.,
$\{2, 3, 6\} = (0, 1, 1, 0, 0, 1, 0) = 0110010$ for $v = 7$.
The cardinality of a subset (the number of 1's
in the corresponding tuple) will be referred to as its size.
The set of all subsets of $V$
is denoted by $2^V$,
which forms a group isomorphic to $\mathbb{Z}_2^v$,
with the symmetric difference
as the group operation.
The symmetric difference corresponds to the bitwise modulo-$2$ addition of the characteristic $v$-tuples,
and we will use $\oplus$ as the symbol for this operation.
In many cases, we will omit this symbol
, i.e., $XY:=X\oplus Y$.
For every $i\in V$, we denote $x_i:=\{i\}$.
Therefore, every
$X=\{i_1,i_2,\ldots,i_w\}\in2^V$ can be written as
$X=x_{i_1}\oplus x_{i_2}\oplus\cdots\oplus x_{i_w}
=x_{i_1}x_{i_2}\ldots x_{i_w}$.

By a \emph{$[t]$-trade} we mean a pair $T=(T_+, T_-)$ of disjoint collections
of $2^{V}$ such that for every $i\in [t]$, $[t]:=\{0,\ldots,t\}$,
every $i$-subset of $V$ is included in the same number
of elements of $T_+$ and of $T_-$.
The sets $T_+$ and $T_-$ are called the \emph{legs} of $T$
and the elements of $T_+$ and $T_-$ are referred to as the {\em blocks} of $T$.
A trade is called  \emph{simple} if it has no repeated blocks;
in that case, $T_+$ and $T_-$ can be considered as ordinary sets.
The
cardinality of a leg (which is, trivially, the same for both legs)
is called the \emph{volume} of $T$, denoted by $\vol(T)$.
  The {\em foundation} of $T$, denoted by $\found(T)$,
  is the set of all $\ell\in V$ such that $\ell$ appears in some blocks of $T$.
For any $\ell\in\found(T)$, the \emph{replication} of $\ell$ is defined as
 $$r_\ell:=|\{B\in T_+: \ell\in B\}|= |\{B\in T_-: \ell\in B\}|.$$
 We use the same notation for the subsets $\alpha\subset\found(T)$ with $|\al|\le t$:
$$r_\al:=|\{B\in T_+: \al\subseteq B\}|=|\{B\in T_-: \al\subseteq B\}|.$$
The trade of volume $0$ is called \emph{void}.
A $[t]$-trade $(T'_+,T'_-)$ is said to be a \emph{$[t]$-subtrade} of a $[t]$-trade $(T_+,T_-)$ if $T'_+\subseteq T_+$ and $T'_-\subseteq T_-$.
An element $\ell$ is said to be \emph{essential} for a trade $T$ if $T$ has a block containing $\ell$ and a block not containing $\ell$.

A trade can be treated as a $\mathbb Z$-valued function over $2^V$, and written as
\begin{equation}\label{eq:pol}
T= \sum_{X\in 2^V} \tau_X X,
\end{equation}
where the positive coefficients $\tau_X$ equal the multiplicity of $X$ in $T_+$,
and the negative coefficients $\tau_X$ equal minus the multiplicity of $X$ in $T_-$.
In terms of such functions, (\ref{eq:pol}),
the definition of a $[t]$-trade can be rewritten as
\begin{equation}\label{eq:def}
 \sum_{X \supseteq S} \tau_X = 0,\qquad \mbox{for every $S\in 2^V$ such that $|S|\le t$.}
\end{equation}
Below, we formally consider summation and multiplication of functions in form (\ref{eq:pol}),
using the rules of the group ring $\mathbb Z[(2^V,\oplus)]$.
This language is convenient for the representation of the trades of small volumes.

A subset $T$ of $2^{V}$ is said to be a \emph{$[t]$-unitrade} if for every subset $S$ of $V$ with $|S|\le t$,
the number of blocks of $T$ including $S$ is an even number.
A $[t]$-unitrade has necessarily an even number of blocks.
If $(T_+, T_-)$ is a simple $[t]$-trade, then clearly $T_+\cup T_-$ is a $[t]$-unitrade.
We extend the definition of volume, foundation and replication to include unitrades $T$
by
$$\vol(T):=|T|/2,~~   r_\ell:=|\{B\in T: \ell\in B\}|/2,$$
and similarly for subsets of $\found(T)$.

\begin{rem}
There is no reason to extend the concept of unitrade to multisets.
Indeed, increasing or decreasing the multiplicity of any block by $2$ does not change the
$[t]$-unitrade property of the multiset.
So, any multiset $M$ is a (generalized) $[t]$-unitrade if and only if
$\mathrm{odd}(M)$ is a $[t]$-unitrade,
where $\mathrm{odd}(M)$
is the set of blocks with odd multiplicity in $M$.
In particular, for any $[t]$-trade $(T_+, T_-)$,
the set $\mathrm{odd}(T_+\uplus T_-)$ is a $[t]$-unitrade, where `$\uplus$' denotes union of multisets.
\end{rem}

\subsection{The binary vector space, Boolean functions and polynomials}\label{s:bool}

The set $2^V$ with the addition operation $\oplus$
and the natural scalar multiplication by $0$ and $1$
is a $v$-dimensional vector space over the Galois field $\mathrm{GF}(2)=(\{0,1\},\oplus)$.
Every subset $S$ of $2^V$ can be represented by the characteristic $\{0,1\}$-function over $2^V$
(such functions are known as \emph{Boolean functions}),
which, in  turn, is uniquely represented as a polynomial of degree at most $v$
in the vector coordinates $y_1,\ldots,y_v$ in the standard basis $x_1=\{1\},\ldots,x_v=\{v\}$,
over $\mathrm{GF}(2)$.
We will say that this polynomial is \emph{associated} with the set $S$.

The set of all $\{0,1\}$-functions on $2^V$ represented by polynomials of degree at most $m$
is denoted by $\RM(m,v)$ (in coding theory, this is known as the Reed--Muller code of order $m$).

\section{Preliminary lemmas}\label{s:prel}

In this section we establish some basic facts about $[t]$-trades which will be used in the rest of the paper. We start with a result
which reveals the connection between $[t]$-trades  and Reed--Muller codes. Consider $f(y_1,\ldots,y_6)=y_1y_2y_3+y_1y_2y_4\in\RM(3,6)$. The set of ones of $f$ is $$T=\{111000,111001,111010,111011,110100,110101,110110,110111\}.$$
It is easily seen that $T$ is indeed a $[2]$-unitrade. This is an example of the following general fact.
\begin{lem}\label{l:unitrade-RM}
The subsets of $2^V$ associated with the polynomials from $\RM(m,v)$, $m<v$,
are exactly the $[t]$-unitrades with $t=v-m-1$.
\end{lem}
\begin{proof} We divide the argument in three parts.

(i) Consider a monomial
$f=y_{i_1}\cdots y_{i_\ell}$, and let $T$ be the set of ones of $f$.
Given a subset $S$ of $V$, we count the number of the members of $T$ `including' $S$  (in terms of tuples, having $1$'s in all positions from $S$).
For a binary vector $\ba=a_1\ldots a_v$, we have $f(\ba)=1$ and $\ba$ includes $S$ if and only if $a_i=1$ for all $i\in S \cup \{i_1,\ldots,i_\ell\}$.
So the number of members of $T$ including $S$ is $2^{|V \setminus\left(S \cup \{i_1,\ldots,i_\ell\}\right)|}$.
This number is even if and only if
$V \setminus\left(S \cup \{i_1,\ldots,i_\ell\}\right)$ is nonempty.

(ii) In particular, if $l\le m$, then $2^{|V \setminus\left(S \cup \{i_1,\ldots,i_\ell\}\right)|}$ is even for every $S$ of size $|S|\le t=v-m-1$.
So, for every monomial of degree less than $v-t$,
the associated set is a $[t]$-unitrade.
This extends to every polynomial of degree less than $v-t$ (i.e., at most $m$),
because any linear combination over $\mathrm{GF}(2)$ preserves the parity properties defining a $[t]$-unitrade.

(iii) On the other hand,
if the degree $s$ of a polynomial is $v-t$ or more,
then it includes
some monomial $y_{i_1}\cdots y_{i_s}$ with coefficient $1$
and does not meet the definition of a $[t]$-unitrade
with $S=V\backslash \{i_1, \ldots ,i_s \}$, $|S|\le t$.
Indeed, by the `only if' statement of (i), for this monomial,
the set $T$ of ones has odd number of elements including $S$;
on the other hand,
for every other monomial of degree at most $s$ this number is even, by
the `if' statement of (i); hence, for the whole polynomial, it is odd.
\end{proof}

In view of Lemma~\ref{l:unitrade-RM}, the next claim is just
the well-known fact on Hamming distance of $\RM(m,v)$ (see, e.g., \cite[Theorem~3 in 13.3]{MWS}),
which is easy to prove by induction on $t$.
\begin{lem}\label{l:odd}
If $T$ is a nonempty $[t]$-unitrade, then $|T|\ge2^{t+1}$, i.e., $\vol(T)\ge2^t$.
\end{lem}

The same bound holds for $[t]$-trades.  The following lemma gives the structure of $[t]$-trades with the minimum volume. A version of this result for $t$-trades is quite well-known, but it can be easily generalized to $[t]$-trades.

\begin{lem}[\cite{FranklPach:83,Hwang:86}]\label{l:min}
The minimum volume of a non-void $[t]$-trade is $2^t$.
Every $[t]$-trade of volume $2^t$ has the form
$$ X_0(X_1-Y_1)(X_2-Y_2) \cdots (X_{t+1}-Y_{t+1}), $$
where $X_0,X_1, \ldots ,X_{t+1},Y_1, \ldots ,Y_{t+1}$
are pairwise disjoint subsets of $V$
and $X_iY_i$ is nonempty for every $i=1, \ldots ,t+1$.
\end{lem}

For $Y\in 2^V$ and a function $T:2^V\to \mathbb Z$,
we call $Y T$ the \emph{$Y$-shift},
or simply a \emph{shift} of $T$.

\begin{example}\label{ex:1}
The function
$$ x_1x_2x_3((1-x_1)(1-x_2)+(1-x_1x_2)(1-x_3)) = 1-x_1x_2-x_2x_3-x_1x_3+2x_1x_2x_3.$$
is a $[1]$-trade of volume $3$.
The left part of the equation represents the trade
as the sum of two simple $[1]$-trades of volume $2$
shifted by $Y=\{1,2,3\}$.
\end{example}

\begin{lem}[\cite{Kro:gaps}]\label{l:shift}
Any shift of a $[t]$-trade is also a $[t]$-trade.
\end{lem}

Given a trade $T$
in the form (\ref{eq:pol})
and an element $i\in V$,
by the \emph{$i$-projection},
or simply a \emph{projection},
of $T$ we  mean
the function $T^i$ obtained from
$T$ by removing $i$ from every block that contains $i$.
Hence, $T^i=P+P'$, where
$T=P+x_iP'$ and $i$ does not occur in $P$ and $P'$.

Note that after a projection, it is possible that two blocks cancel out each other,
 so the volume can be reduced. If the volume of $T$
equals the volume of $T^i$,
then we  say that
$T$ is an \emph{extension} of $T^i$.
So, an extension of a $[t]$-trade $T$
is a $[t]$-trade obtained from $T$
by including a new element in some blocks of $T$.

\begin{example}\label{ex:2}
The following simple $[1]$-trade
is an extension of the $[1]$-trade from Example~\ref{ex:1}:
$$1-x_1x_2-x_2x_3+x_1x_2x_3-x_1x_3\underline{x_4}+x_1x_2x_3\underline{x_4}.$$
\end{example}

The following four lemmas are straightforward from the definitions.

\begin{lem}\label{l:proj}
A projection of a $[t]$-trade is a $[t]$-trade. \end{lem}

\begin{lem}\label{l:sub}
Let $T = P+x_i P'$ be a $[t]$-trade,
where $i$ does not occur in the blocks of $P$, $P'$.
Then $P$, $P'$, and $x_i P'$ are $[t-1]$-trades.
\end{lem}

\begin{lem}\label{l:[1]}
 If $(T_+,T_-)$ is a $[1]$-trade, then $\bigoplus_{X\in T_+\cup T_-} X = \emptyset$.
\end{lem}

\begin{lem}\label{l:dbl}
If $P$ is a $[t-1]$-trade and
the element $i$ does not occur in its blocks, then $(1-x_i)P$ is a $[t]$-trade.
\end{lem}

We  say that a $[t]$-trade is $s$-\emph{small}
for some $s>1$ if its volume is less than $s\cdot 2^t$.
The $2$-small trades will be referred to as \emph{small}.

The following  statement
plays an important role
in the computer-aided classification
of small $[t]$-trades.
\begin{cor}\label{c:decomp}
For each $i$ from $V$, every $[t]$-trade $T$
is decomposable to the sums
\begin{eqnarray}
T &=& x_i T^i+(1-x_i)P \label{eq:t+P} \\
&=&     T^i-(1-x_i)P' \label{eq:t+P'},
\end{eqnarray}
where $T^i$ is a $[t]$-trade,
$P$ and $P'$ are $[t-1]$-trades,
and the element $i$ does not occur
in $T^i$, $P$, $P'$.
Moreover,
if $T$ is an $s$-small $[t]$-trade for some $s$,
then $T^i$
is an $s$-small $[t]$-trade
and one of $P$, $P'$
is an $s$-small $[t-1]$-trade.
\end{cor}
\begin{proof}{
If we present the $[t]$-trade
in the form $T= P+x_iP'$
and define
$T^i=P+P'$
to be the $i$-projection
of $T$,
then the first statement
trivially follows from
Lemmas~\ref{l:sub} and~\ref{l:proj}.
The volume of the projection
is trivially not greater than
the volume of the original trade;
so, if $T$ is $s$-small
then so is $T^i$.
Moreover, the volume of $T$
is the sum
of the volumes of $P$ and $P'$;
so, if it is less than $s\cdot 2^{t}$,
then one of the summands
is less than $s\cdot 2^{t-1}$,
which means that
the corresponding $[t-1]$-trade
is $s$-small.}
\end{proof}

As mentioned before, the minimum distance of $\RM(m,v)$ is $d=2^{v-m}$.  Kasami and Tokura \cite{KasamiTokura:70} characterized codewords of $\RM(m,v)$
with weight at most $2d$. This result is the base of our characterization of $[t]$-trades with small volumes.
\begin{lem}[\cite{KasamiTokura:70}] \label{kasami}
Any Boolean function $f$ from $\RM(m,v)$ of weight greater than $2^{v-m}$
and less than $2\cdot2^{v-m}$ can be reduced by an invertible affine transformation of its variables
to one of the following forms:
\begin{align*}
 f(y_1,\ldots, y_v)& = y_1 \cdots y_{m-\mu} \cdot (y_{m-\mu+1} \cdots y_m \oplus y_{m+1} \cdots y_{m+\mu}),\tag{A}\\
 f(y_1,\ldots, y_v)& = y_1 \cdots y_{m-2} \cdot (y_{m-1}\cdot y_m \oplus y_{m+1}\cdot y_{m+2} \oplus \cdots \oplus y_{m+2\nu-3}\cdot y_{m+2\nu-2}),\tag{B}
\end{align*}
where $v\ge m + \mu$, $m\ge\mu\ge2$, $v\ge m+ 2\nu-2 $ and $\nu\ge3$.
Any Boolean function
from $\RM(m,v)$ of minimum nonzero weight, $2^{v-m}$, is the characteristic function of a
$(v-m)$-dimensional affine subspace of $2^V$.
\end{lem}

Based on Lemma~\ref{l:unitrade-RM} and the Kasami--Tokura characterization, the gaps conjecture
was proved in \cite{Kro:gaps} in the more general setting of $[t]$-unitrades.
For future reference,  we state it as the following lemma.

\begin{lem} \label{gaps} If $T$ is a nonempty $[t]$-unitrade  with $\vol(T)<2^{t+1}$, then
$$\vol(T)\in\left\{2^t,\left(2-\frac{1}{2}\right)2^t,\ldots,\left(2-\frac{1}{2^t}\right)2^t\right\}.$$
In particular, the same holds for simple $[t]$-trades.
\end{lem}

\begin{lem} \label{affinesAREtrades}
Every $(t+1)$-dimensional affine subspace of $2^V$ is a $[t]$-unitrade.
\end{lem}
\begin{proof}
Let $A$ be a $(t+1)$-dimensional affine subspace of $2^V$.
Let $\{i_1,\ldots,i_r\}\subset V$ with $r\le t$.
Consider the $(v-r)$-dimensional affine subspace
$W=\{(y_1,\ldots,y_v): y_{i_1}=\cdots=y_{i_r}=1\}$.
Then $W\cap A$ is either empty or it is  an affine subspace of $2^V$
of dimension at least
$(v-r)+(t+1)-v \ge 1$
and so it has an even cardinality.
Considering the vectors of $A$ as subsets of $V$, this means that $\{i_1,\ldots,i_r\}$ is contained in an even number of blocks of $A$.
\end{proof}

\begin{lem} \label{unitrade} If $T$ is a nonempty $[t]$-unitrade, then $\langle T\rangle\setminus T$ is also a $[t]$-unitrade, where $\langle T\rangle$ denotes the affine span of $T$.
\end{lem}
\begin{proof}{%
Let $d$ be the dimension of $\langle T\rangle$.
By Lemma~\ref{gaps}, $|T|\ge2^{t+1}$.
Therefore, $d\ge t+1$, and hence by Lemma~\ref{affinesAREtrades},
$\langle T\rangle$ is a $[t]$-unitrade.
It follows that  $\langle T\rangle\setminus T$ is also a $[t]$-unitrade.
}\end{proof}

\begin{lem}\label{T_x}
Let $T=(T_+,T_-)$ be a $[t]$-trade. Let $\alpha,\beta\subset\found(T)$ with $\al\cap\beta=\emptyset$.
Consider
$$
 R^+={\{B\in T_+: \al\subset B,\,\beta\cap B=\emptyset\}},
 \qquad
 R^-={\{B\in T_-: \al\subset B,\,\beta\cap B=\emptyset\}},
 $$
 as multisets.
Then $(R_+,R_-)$ is a $\left(t-|\al|-|\beta|\right)$-trade.
\end{lem}
\begin{proof}
The case $|\al|+|\beta|=1$ is done by Lemma~\ref{l:sub}.
The general case is proven by induction on $|\al|+|\beta|$.
\end{proof}

We denote the trade $(R_+,R_-)$ of Lemma~\ref{T_x} by $T_{\al\ov\beta}$. In particular, we use the notation $T_i$ for $\al=\{i\}$ and $\beta=\emptyset$  and $T_{\ov j}$ for $\al=\emptyset$ and $\beta=\{j\}$.

We call a $[t]$-trade $T$ {\em reduced} if
$$r_i\le\tfrac{1}{2}\vol(T),\qquad \hbox{for all}~i\in\found(T).$$

\begin{lem}\label{l:reduced}
Every $[t]$-trade can be transformed by some shifts into a reduced $[t]$-trade.
\end{lem}
\begin{proof}
Let $T$ be a $[t]$-trade, and let $I$ consist of all $i$'s such that $r_i > \tfrac{1}{2}\vol(T)$.
In $I\oplus T$,  the $I$-shift of $T$, the replication of $i$ is $\vol(T)-r_i < \tfrac{1}{2}\vol(T)$ for every $i\in I$ (the replications of elements in $V\setminus I$ remains the same).
It follows that $I\oplus T$ is reduced.
\end{proof}

\section[Affine rank of simple {[t]}-trades]{Affine rank of simple $[t]$-trades} \label{s:rank}

Recall that by Lemma~\ref{l:unitrade-RM}, unsigned simple $[t]$-trades with a foundation of size $v$ can be regraded as codewords of the Reed--Muller code $\RM(v-t-1,v)$.
As given in Lemma~\ref{kasami}, the codewords of Reed--Muller codes with weights at most twice the minimum distance have been characterized in
 \cite{KasamiTokura:70} and subsequently divided into Types (A) or (B).
Accordingly, simple $[t]$-trades (and also $[t]$-unitrades) with volume at most $2^{t+1}$ can be categorized into Types (A) or (B).
Krotov \cite{Kro:gaps} considered this possible dichotomy and put forward the existence of $[t]$-trades of Type\,(B) as an open problem.
In this section we establish some results about the affine rank of trades from which it follows that trades of Type\,(B) do not exist.
In addition, the non-existence of simple $[t]$-trades with volumes $2^{t+1}+2^i$, $(t-1)/2\le i\le t-4$ is also established.

We denote the affine rank (the dimension of the affine span)
of a subset $S$ of the vector space $2^{V}$  by $\af(S)$.
If $T=(T_+,T_-)$ is a simple $[t]$-trade, by $\af(T)$ we mean
$\af(T_+\cup T_-)$.

We first show  how the types of $[t]$-trades can be distinguished by means of their affine rank.
\begin{pro} \label{affine}
Let $T$ be a simple $[t]$-trade with $\vol(T)=2^{t+1}-2^i$ for $i\in\{0,1,\ldots,t-1\}$.
\begin{itemize}
  \item[\rm(i)] If $T$ is of Type\,{\rm(A)}, then $\af(T)=2t+2-i$.
  \item[\rm(ii)] If $T$ is of Type\,{\rm(B)}, then $(t-1)/2\le i\le t-2$ and  $\af(T)=t+3$.
\end{itemize}
In particular, if either $\af(T)\ge t+4$, $i=t-1$ or $i<(t-1)/2$, then $T$ is of Type\,{\rm(A)}.
\end{pro}
\begin{proof} Let $T'$ denote the corresponding $[t]$-unitrade with $T$. Note that an invertible affine transformation of the variables does not change the affine rank and the cardinality of the set of ones of the polynomials given in Lemma~\ref{kasami}. So we may assume that $T'$ is the set of ones of such polynomials.

(i)
 Considering the associated polynomial of $T'$ given by Lemma~\ref{kasami}\,(A), it is seen that
$T'$ is the symmetric difference of two intersecting affine subspaces of dimension
$t+1$. If the dimension of the intersection is $i$, $0\le i<t$, then the cardinality of $T'$ is $2^{t+2}-2^{i+1}$
and its affine rank is $2t+ 2-i$.

(ii)
 $T'$ is the set of ones of the polynomial given by Lemma~\ref{kasami}\,(B). By a counting argument, we have
\begin{align*}
|T'|&=2^{v-m-2\nu+2}\sum_{j\,{\rm odd}} \binom{\nu}{j} 3^{\nu-j}\\
&=2^{v-m-2\nu+2}\cdot\frac{1}{2}\big((3+1)^\nu-(3-1)^\nu\big)\\
&=2^{t+2}-2^{t+2-\nu}~~~~\hbox{(as $t=v-m-1$).}
\end{align*}
We have $\nu\ge3$ and $v\ge m-2 + 2\nu$, so $3\le \nu\le(t+3)/2$. As $|T'|=2\vol(T)=2^{t+2}-2^{i+1}$, it follows that $i=t+1-\nu$ and thus
 $(t-1)/2\le i\le t-2$.

A unitrade of Type\,(B) is an intersection of an affine subspace of dimension $t+3$ and the set of ones
of a quadratic function. So $\af(T')\le t+3$. If $\af(T')\le t+2$, then by Lemma~\ref{unitrade}, $\langle T'\rangle\setminus T'$
is $[t]$-unitrade with volume $2^i$ for some $0\le i\le t-1$ which is a contradiction to Lemma~\ref{gaps}. It follows that $\af(T')=t+3$.
\end{proof}

From Lemma~\ref{kasami} it is clear that $[t]$-unitrades of Type\,(B) (and so with affine rank $t+3$) do exist. However, we manage to prove that this is not the case for $[t]$-trades. It follows that unitrades of Type\,(B) are not `splittable.'  This means that, although an unsigned $[t]$-trade gives a $[t]$-unitrades, but this is not reversible in general.

\begin{lem}\label{rx>2t-1}
Let $t\ge3$ and $T$ be a simple $[t]$-trade such that for all $i\in\found(T)$, $r_{i}=2^{t-1}$. If $\vol(T)>2^t$, then  $\af(T)\ge t+4$.
\end{lem}
\begin{proof}{ Suppose that $\vol(T)>2^t$. So by Lemma~\ref{gaps},  $\vol(T)\ge1.5\cdot2^t$.
For any $i \in \found(T)$, $T_{i}$ is a $[t-1]$-trade of volume $r_{i}$.
Choose $i , j \in \found(T)$ so that $r_{i j}\not\in\{0,2^{t-1}\}$.
Then $r_{i j}=2^{t-2}$. As $T_{i j}$ is a $[t-2]$-trade of minimum volume and $t\ge3$, there exists some
 $k\in\found(T)$ with $r_{i j k}=2^{t-3}$.
 It turns out that $r_{i k},r_{j k}\not\in\{0,2^{t-1}\}$ and so
 $r_{i k}=r_{j k}=2^{t-2}$.
Then
\begin{align*}
\vol(T_{\ov{i j k}})&=\vol(T)-\vol(T_{i})-\vol(T_{j})-\vol(T_{k})+\vol(T_{i j})+\vol(T_{i k})+\vol(T_{j k})-\vol(T_{i j k})\\
&\ge1.5\cdot2^t-3\cdot2^{t-1}+3\cdot2^{t-2}-2^{t-3}= 1.25\cdot 2^{t-1} >2^{t-1}.
\end{align*}
It follows that $T_{\ov{i j k}}$ has affine rank at least $t+1$. On the other hand, as
$\vol(T_{i\ov{j k}})=\vol(T_{j\ov{i k}})=\vol(T_{k\ov{i j}})=2^{t-3}\ne0$,
there are three more affinely independent vectors in $T$ each containing exactly one of $i$, $j$ or $k$.
This means that the affine rank of $T$ is at least $t+4$.
}\end{proof}

\begin{lem}\label{rem:t-4}
Let $T$ be a $[t]$-unitrade.
\begin{itemize}
 \item[\rm (i)] If
$\vol(T)=2^{t+1}\pm 2^i$, $(t-1)/2 \le i \le t-1$, and $\af(T)=t+3$,
then the associated polynomial
corresponding to $T$ can be obtained from
  \begin{equation}\label{fTypeB}
f(y_1,\ldots,y_v)= y_1\cdots y_{m-2}\cdot (y_{m-1}\cdot y_m \oplus y_{m+1}\cdot y_{m+2} \oplus \cdots \oplus y_{m+2\nu-3}\cdot y_{m+2\nu-2}\oplus a)
\end{equation}
  by an invertible affine transformation of variables,
  where $m=v-t-1$,  $\nu=t+1-i$, $a=1$
  if $\vol(T)=2^{t+1}+2^i$
  and $a=0$ if $\vol(T)=2^{t+1}-2^i$.

 \item[\rm (ii)]
 If $2^{t+1}<\vol(T)<2^{t+1}+2^{t-3}$, then $\vol(T)=2^{t+1}+2^i$, for some $i$,
 $(t-1)/2 \le i \le t-4$,
the associated polynomial to $T$ is
of the form \eqref{fTypeB} with $a=1$,
and $\af(T)=t+3$.
\end{itemize}
\end{lem}
\begin{proof}
 (i)
 If $\af(T)=t+3$, then there is an invertible affine variable transformation
 that sends $T$ to a $[t]$-unitrade $T'$ whose affine span is defined by the equations
 \begin{equation} \label{eq:=1=1} y_1=1,\  \ldots,\ y_{m-2}=1 \quad (m=v-t-1). \end{equation}
 It follows from \eqref{eq:=1=1} that the polynomial associated to $T'$ has the form
 $$ g(y_1,\ldots,y_v) =  y_1\cdots y_{m-2}\cdot h(y_{m-1},\ldots,y_v).$$
 By  Lemma~\ref{l:unitrade-RM}, $g$ has degree at most $m$, and hence $h$ has degree at most $2$.
 The polynomial $h$, as a polynomial in the $t+3$ variables $y_{m-1}$, \ldots, $y_v$,
 has $2\vol(T)$ ones, which is either $2^{t+2}- 2^{i+1}$ or $2^{t+2}+ 2^{i+1}$.
 By the results of \cite{SloBer:1970}, $h$ is affinely equivalent to
 $$ y_{m-1}\cdot y_m \oplus y_{m+1}\cdot y_{m+2} \oplus \cdots \oplus y_{m+2\nu-3}\cdot y_{m+2\nu-2} \oplus a $$
 with $a=0$ or $a=1$, respectively.
 Therefore, $g$ is affinely equivalent to $f$ in \eqref{fTypeB}.

 (ii) is straightforward from Lemma~\ref{l:unitrade-RM}
 and the characterization of the codewords of $\RM(m,v)$
 of weight smaller
 than $2.5\cdot 2^{m-v}$ \cite{KTA:76}.
\end{proof}

\begin{example}\label{ex:3}
The set
$T=\{00111,10011,01011,11001,11100,11010\}$
is a $[1]$-unitrade of volume $2^2-2^0$.
Consider the linear transformation $f$ between two 4-dimensional linear spaces that maps
$ 00000 \to 00000 $,
$ 10100 \to 00010 $,
$ 01100 \to 00001 $,
$ 11110 \to 01111 $,
$ 11011 \to 01011 $,
$ 11101 \to 00111 $. Note that $\{10100, 01100, 11110, 11011\}$ and $\{00010, 00001, 01111, 01011\}$ are linearly independent sets and
$ 11101=10100+01100+11110+11011 \to 00111=00010+00001+01111+01011 $.
Now $\x\mapsto f(\x+00111)+11100$ is the invertible affine transformation that maps
$ 00111 \to 11100 $,
$ 10011 \to 11110 $,
$ 01011 \to 11101 $,
$ 11001 \to 10011 $,
$ 11100 \to 10111 $,
$ 11010 \to 11011 $. So $T$ is mapped
onto the $[1]$-unitrade
$\{11100,11110,11101,10011,10111,11011\}$
which is the set of ones of the polynomial
$ y_1\cdot(y_2\cdot y_3 \oplus y_4 \cdot y_5)$
of type \eqref{fTypeB}.
 \end{example}

\begin{lem}\label{rx=vol/2} Let $T$ be a $[t]$-unitrade with
$\af(T)=t+3$ and $\vol(T)=2^{t+1}\pm2^i$
where $t/2\le i\le t-1$.
 Then $r_j=\vol(T)/2$ for some $j\in\found(T)$.
 \end{lem}
\begin{proof}{By Lemma~\ref{rem:t-4},
the associated polynomial corresponding to $T$ can be obtained from \eqref{fTypeB} by an invertible affine transformation of variables.
 We have
  $$m+2\nu-2=v-1+t-2i\le v-1.$$
 So $y_v$ is a free variable of $f$, which implies that $r_v=\vol(T)/2$.
 In fact the set of ones of $f$ is of the form $S\times\{0,1\}$ for some $S\subset 2^{[v-1]}$ with $|S|=\vol(T)$.
 Let $\y\mapsto \y M+\bb$ be the invertible affine transformation which gives the associated polynomial of $T$. 
 Hence $T$ is the set of ones of $g(\y)=f(\y M+\bb)$, i.e., 
  $$T=\{(\y-\bb)M^{-1}: \y\in S\times\{0,1\}\}=\{\x M^{-1}: \x\in S'\times\{0,1\}\},$$
  for some $S'$ with $|S'|=|S|$.
  The last row of $M^{-1}$ should be nonzero. So we may assume that  the $j$-th column of $M^{-1}$, say $\ba^\top$ has its  last component equal to $1$.
  Then we have either $\x\ba^\top=1$ for all $\x\in S'\times\{0\}$ or $\x\ba^\top=1$ for all $\x\in S'\times\{1\}$.
  This means that $r_j=|S'|/2=\vol(T)/2$.
}\end{proof}

\begin{lem} \label{found=afrk} For any simple $[t]$-trade $T$, there exists a simple $[t]$-trade $T'$ with $|\found(T')|=\af(T')=\af(T)$ and $\vol(T')=\vol(T)$.
\end{lem}
\begin{proof}
Denote by $A$ the affine span of $T$, and by $A_i$, the $i$-projection of $A$.
If $|A_i|<|A|$ for all $i\in\found(T)$, then $A=2^{\found(T)}$, and the statement trivially holds with $T'=T$.
Otherwise, $|A_i|=|A|$ for some $i\in \found(T)$, and the $i$-projecting acts bijectively on $A$.
It follows that the $i$-projection of $T$ has the same volume and affine rank as $T$, but smaller foundation.
Repeating this operation  $|\found(T)|-\af(T)$ times, we find a required $T'$.
\end{proof}

\begin{lem}\label{l:computational} The simple $[2]$-trades of foundation size $5$ and volume $6$, $8$, $10$ satisfy the following properties.
\begin{itemize}
  \item[\rm(i)] In any simple $[2]$-trade with volume $6$ and foundation size $5$, the number of elements with odd replication (the only possible odd value is $3$) is odd.
  \item[\rm(ii)] In any simple $[2]$-trade with volume $8$ and foundation size $5$, the number of elements with odd replication ($3$ or $5$) is even.
  \item[\rm(iii)] In any simple $[2]$-trade with volume $10$ and foundation size $5$, the number of elements with odd replication (the only possible odd value is $5$) is odd.
\end{itemize}
\end{lem}
The proof of Lemma~\ref{l:computational} is by computation and will be addressed in Section~\ref{s:comp}.
The sharpening claims in the parenthesis can be easily shown theoretically,
but we will not use them in the further discussion.

\begin{lem}\label{t=2,3}
Let $t=2,3$, and $T$ be a simple $[t]$-trade with $1.5\cdot2^t<\vol(T)<2.5\cdot2^t$ and $\vol(T)\ne2\cdot2^t$.
Then the affine rank of $T$ is at least $t+4$.
\end{lem}
\begin{proof}{As shifts do not change the volume and affine rank of trades, in view of Lemma~\ref{l:reduced}, we may assume that $T$ is a reduced simple $[t]$-trade.

First let $t=2$. Then $\vol(T)=7$ or $9$.
If $\vol(T)=7$, then by Proposition~\ref{affine}, it  has affine rank $6$ (this is even true for $[2]$-unitrades of volume 7).
Let $\vol(T)=9$.  We have $\af(T)\ge\left\lceil\log_2(2\vol(T))\right\rceil\ge5$.  If $\af(T)=5$, by Lemma~\ref{unitrade}, there exists a $[2]$-unitrade with affine rank $5$ and $2^5-18=14$ blocks, which cannot exist as just shown.
It follows that $\af(T)\ge6$.

Now assume that $t=3$. We have $\vol(T)\in\{14,15,17,18,19\}$.
If $\vol(T)=15$, we are done by Proposition~\ref{affine}.
Let $\vol(T)$ is $17$ (respectively, $19$). We have $\af(T)\ge\left\lceil\log_2(2\vol(T))\right\rceil\ge6$.
If $\af(T)=6$, then by Lemma~\ref{unitrade},
there exists a $[3]$-unitrade with affine rank $6$ and cardinality $13$ (respectively $15$)
 which is impossible by Lemma~\ref{gaps} (by the above argument).
 So $\af(T)\ge7$.
It remains to prove the assertion for volumes $14$ and $18$.

Suppose  $\vol(T)=14$. For a contradiction, let $\af(T)=6$.
By Lemma~\ref{found=afrk}, we may assume that $|\found(T)|=6$.
Applying Lemma~\ref{gaps} to $T_{i}$ we obtain  $r_{i}\in\{4,6,7\}$ for all $i\in\found(T)$.
If $r_{i}=7$ for some $i\in\found(T)$, then $T_{i}$ is a $[2]$-trade of volume $7$ and has affine rank at least $6$ by Proposition~\ref{affine}. Hence $\af(T)\ge7$, a contradiction.
Hence for all $i\in\found(T)$, $r_{i}=4$ or $6$. If for all $i\in\found(T)$, $r_{i} =4$, then we are done by 
Lemma~\ref{rx>2t-1}. So assume that $r_{i}=6$ for some $i\in\found(T)$.
Here $T_{i}$ is a $[2]$-trade of volume $6$ and $|\found(T_{i})|=5$ ($|\found(T_{i})|$ cannot be smaller than $5$ as $\af(T_{i})=5$).
Note that $\vol(T_{\ov i})=8$. Also $|\found(T_{\ov i})|=5$, because $T_{\ov i j}$ is a $[1]$-trade and so $\af(T_{\ov i j})\ge4$, it follows that
$\af(T_{\ov i})\ge5$. On the other hand, $\af(T_{\ov i})\le\af(T)-1=5$.
Our aim is to obtain a contradiction by considering the replications of elements in both $T_{i}$ and  $T_{\ov i}$.
In view of Lemma~\ref{l:computational} applied to $T_i$, the number of $j\in\found(T)$ with $r_{i j}=3$ must be odd.
We further claim that $r_{i j}=3$ if and only if $r_{\ov i j}=3$.
The claim follows from the fact that if either $r_{i j}=3$ or $r_{\ov i j}=3$, then $r_{j}=6$; since otherwise, $r_{j}=4$, and then $T_{i j}$ or $T_{\ov i j}$ would be a $[1]$-trade of volume $1$, a contradiction. Also there are no $k\in\found(T)$ with $r_{\ov i k}=5$; since otherwise $r_{k}$ is necessarily $6$, and so $T_{i\ov k}$ would be a
$[1]$-trade of volume $1$, again a contradiction.
The above argument shows that the number of elements with an odd replication in $T_{i}$ is the same as the number of elements with an odd replication in $T_{\ov i}$. However, by Lemma~\ref{l:computational}, the former is an odd number and the latter is an even number, again a contradiction.

Finally, suppose that $\vol(T)=18$ and $\af(T)=6$.
By Lemma~\ref{found=afrk},
we may assume that $|\found(T)|=6$.
By Lemma~\ref{gaps} and since $T$ is reduced we have $r_{i}\in\{4,6,7,8,9\}$
for all $i\in\found(T)$.
If $r_{i}\in\{7,9\}$,
for some $i\in\found(T)$, then $T_{i}$ is a $[2]$-trade of volume $7$ or $9$ and consequently $\af(T_i)\ge6$ as we just showed.
Hence $\af(T)\ge7$, a contradiction.
So $r_{i}\in\{4,6,8\}$ for all $i\in\found(T)$.

We claim that $r_{k}=8$ for some $k\in\found(T)$.
Otherwise, $r_{i}\in\{4,6\}$ for all $i\in\found(T)$.
If for all $i$, $r_{i}=4$, then by
Lemma~\ref{rx>2t-1} we have that $\af(T)\ge 7$.
 If $r_{i}=6$, for some $i\in\found(T)$,
 then by Lemma~\ref{rx=vol/2} applied to $T_{i}$,
 we obtain that $r_{i j}=3$ for some $j\in\found(T)$.
 It turns out that $r_{j}=6$.
 Thus $T_{\ov{i j}}$ has $18$ blocks;
 so $\af(T_{\ov{i j}})\ge 5$.
 It follows that $\af(T)\ge 7$, a contradiction.
 Hence, the claim follows.

Therefore, we assume that $r_{k}=8$ and so $\vol(T_{\ov k})=10$.
Also $\af(T_{\ov k})=|\found(T_{\ov k})|=5$.
For every $i \in\found(T)$,  $r_{i}$ is even ($4$, $6$, or $8$);
hence, the volumes of $T_{k i}$ and $T_{\ov k i}$ have the same parity.
It follows that the number of elements with an odd replication in $T_{k}$
is the same as the number of elements with an odd replication in $T_{\ov k}$.
However, the former is an odd number by Lemma~\ref{l:computational}(ii)
and the latter is an even number by Lemma~\ref{l:computational}(iii), a contradiction.
}\end{proof}

Now, we are ready to prove the main result of this section.

\begin{thm}\label{main} If $T$ is a simple $[t]$-trade with $1.5\cdot2^t<\vol(T)<2.5\cdot2^t$ and $\vol(T)\ne2^{t+1}$, then the affine rank of $T$ is at least $t+4$.
\end{thm}

\begin{proof}{%
We proceed by induction on $t$.
For $t=1$, there is no trade satisfying the assumptions, and $t=2,3$ has been settled in Lemma~\ref{t=2,3}.
Hence we assume that $t\ge4$.

Since shifts do not change the affine rank of trades, we may assume that $T$ is reduced.
 As $T$ is reduced, $r_{i}\le\vol(T)/2<2.5\cdot2^{t-1}$ for all $i\in\found(T)$.
 If there exists some $i\in\found(T)$ with $r_{i}\ne2^t$ and $r_{i}>1.5\cdot2^{t-1}$,
 then $T_{i}$ is a simple $[t-1]$-trade with $\vol(T_{i})\ne2^t$ and $1.5\cdot2^{t-1}<\vol(T_{i})<2.5\cdot2^{t-1}$.
 So by the induction hypothesis, $\af(T_{i})\ge t+3$.
 Therefore, $\af(T)\ge t+4$, and we are done.
Hence we can assume that
 \begin{equation}\label{r_{i}=2cases}
\hbox{for all $ i \in\found(T)$, either $r_{i}=2^t$  or $r_{i}\le1.5\cdot2^{t-1}$.}
\end{equation}
So it suffices to consider the following two cases.

\noindent{\bf Case 1.} There exist some $ i \in\found(T)$ with $r_{i}=2^t$.

As we assumed that $T$ is reduced, $\vol(T)\ge2r_{i}=2^{t+1}$, so $2\cdot2^{t-1}<\vol(T)-r_{i}<3\cdot2^{t-1}$. If further, $\vol(T_{\ov i})=\vol(T)-r_{i}<2.5\cdot2^{t-1}$, then by the induction hypothesis,
 $\af(T_{\ov i})\ge t+3$, and so we are done. Therefore, we  assume that $\af(T_{\ov i})=t+2$ and $2.5\cdot2^{t-1}\le\vol(T_{\ov i})<3\cdot2^{t-1}$.
 Then by Lemma~\ref{unitrade}, there exists a $[t-1]$-unitrade $T'$ with
 $2^{t-1}< \vol(T')=2^{t+1}-\vol(T_{\ov i})\le1.5\cdot2^{t-1}.$
 By Lemma~\ref{gaps}, $\vol(T')=1.5\cdot2^{t-1}$ and so $\vol(T_{\ov i})=2.5\cdot2^{t-1}$ implying that $\vol(T)=4.5\cdot2^{t-1}$.
 If $\af(T)\le t+3$, then by Lemma~\ref{rx=vol/2}, $r_{j}=4.5\cdot2^{t-2}$ for some $j$, which is impossible in view of \eqref{r_{i}=2cases}.
 So $\af(T)\ge t+4$ and we are done.

\noindent{\bf Case 2.} For all $i\in\found(T)$, $r_{i}\le1.5\cdot2^{t-1}$.

Applying Lemma~\ref{gaps} to $T_{i}$  we obtain that $r_{i}=2^{t-1}$ or $1.5\cdot2^{t-1}$  for all $i\in\found(T)$.
If for all $i\in\found(T)$, $r_{i}=2^{t-1}$, then we are done by  Lemma~\ref{rx>2t-1}. So assume that $r_{i}=1.5\cdot2^{t-1}$ for some $i\in\found(T)$.
It follows that $$1.5\cdot2^{t-1}<\vol(T_{\ov i})=\vol(T)-r_{i}<3.5\cdot2^{t-1},~~\vol(T_{\ov i})\ne2.5\cdot2^{t-1}.$$
Note that we  also have  $\vol(T_{\ov i})\ne2\cdot2^{t-1}$ (since otherwise $\vol(T)=3.5\cdot2^{t-1}$ and so Lemma~\ref{rx=vol/2} implies the existence of some $j\in\found(T)$ with $r_{j}=3.5\cdot2^{t-2}$, hence a contradiction). Therefore, if $\vol(T_{\ov i})<2.5\cdot2^{t-1}$, then $T_{\ov i}$ satisfies the induction hypothesis, and so
$\af(T_{\ov i})\ge t+3$ implying that $\af(T)\ge t+4$.
Now suppose that $\vol(T_{\ov i})>2.5\cdot2^{t-1}$.
Then $\af(T_{\ov i})\ge\left\lceil\log_2(2\vol(T_{\ov i}))\right\rceil=t+2$.
If $\af(T_{\ov i})=t+2$, then  by Lemma~\ref{unitrade},
$T'=\langle T_{\ov i}\rangle\setminus T_{\ov i}$ is a $[t-1]$-unitrade with $\vol(T')<1.5\cdot2^{t-1}$. So by Lemma~\ref{gaps},
$\vol(T')=2^{t-1}$ which in turn implies that $\vol(T)=4.5\cdot2^{t-1}$. Now Lemma~\ref{rx=vol/2} implies the existence of some $j\in\found(T)$ with $r_{j}=4.5\cdot2^{t-2}$, a contradiction. So $\af(T_{\ov i})\ge t+3$ and thus $\af(T)\ge t+4$.
}\end{proof}
Now, by Theorem~\ref{main} and Proposition~\ref{affine} we have the following corollary which answers an open problem of \cite{Kro:gaps}.
\begin{cor}
There do not exist simple $[t]$-trades $T$ with $2^t<\vol(T)<2^{t+1}$ of Type\,{\rm(B)}.
\end{cor}
The following corollary  will be used in the next section.

\begin{cor}\label{c:form(i)}
There do not exist simple $[t]$-trades $T$ with $\vol(T)=2^{t+1}+2^i$ for $(t-1)/2\le i\le t-4$.
\end{cor}
\begin{proof}{Suppose for a contradiction that $T$ is a simple $[t]$-trade with $\vol(T)=2^{t+1}+2^i$, $(t-1)/2\le i\le t-4$. By Theorem~\ref{main}, $\af(T)\ge t+4$.
On the other hand, let $T'$ be the unitrade associated with $T$. By Lemma~\ref{rem:t-4}(ii), $\af(T)=\af(T')= t+3$, a contradiction.
}\end{proof}

\section[Spectrum of volumes of simple {[t]}-trades ...]{Spectrum of volumes of simple $[t]$-trades between $2\cdot2^t$ and $2.5\cdot2^t$} \label{s:spectr}

Based on the characterization of codewords of Reed-Muller code with weights within the range $2$ and $2.5$ times the minimum distance by Kasami {\em et al.} \cite{KTA:76}, the following was obtained in \cite{Kro:gaps}.
\begin{thm}\label{l:Kro2.5} If the volume of a $[t]$-trade is between $2\cdot2^t$ and $2.5\cdot2^t$, then it has one of the following forms:
\begin{itemize}
\item[\rm(i)] $2^{t+1}+2^i$ for $i=\lceil(t-1)/2\rceil,\ldots,t-2$;
   \item[\rm(ii)] $2^{t+1}+2^{t-1}-2^i$ for $i=0,\ldots,t-2$;
  \item[\rm(iii)]  $2^{t+1}+2^{t-1}-3\cdot2^i$ for $i=0,\ldots,t-3$.
\end{itemize}
\end{thm}
In Corollary~\ref{c:form(i)}, we showed that $[t]$-trades with volumes of the form (i) do not exist (except for $i=t-2$ and $t-3$ which can be represented in the form (ii) and (iii), respectively).
In this section, we show by construction that they do exist with volumes of the forms (ii) and (iii). So the spectrum of volumes of $[t]$-trades in the range $2\cdot2^t$ and $2.5\cdot2^t$ is completely determined. For the construction, we employ the following observation of \cite{Kro:gaps}.

\begin{lem}  \label{T1DeltaT2}
 Assume that $(T_+, T_-)$ and $(T'_+, T'_-)$ are two different simple $[t]$-trades such that $T_+\cap T'_+=T_-\cap T'_-=\emptyset$. Then
 $\big((T_+ \cup T'_+)\setminus(T_- \cup T'_-),\, (T_- \cup T'_-)\setminus(T_+ \cup T'_+)\big)$ is a
  simple $[t]$-trade.
\end{lem}

\begin{thm} \label{TradeContruaction} There exist simple $[t]$-trades of volumes:
\begin{itemize}
  \item[\rm(i)] $2^{t+1}+2^{t-1}-2^i$ for $i=0,\ldots,t-2$;
  \item[\rm(ii)]  $2^{t+1}+2^{t-1}-3\cdot2^i$ for $i=0,\ldots,t-3$.
\end{itemize}
 \end{thm}
\begin{proof}{(i) Let
\begin{align*}
   T_1&:=\left\langle\{1\},\ldots,\{t+1\}\right\rangle,\\
   T_2&:=\langle\{1\},\ldots,\{t-1\},\{t+2\},\{t+3\}\rangle.
\end{align*}
Define $T^+_1$ ($T^-_1$) to be the set of vectors of $T_1$ with an odd (even) weight  and $T^+_2$ ($T^-_2$) to be the set of vectors of $T_2$ with an even (odd) weight.
We have $T^+_1\cap T^+_2=T^-_1\cap T^-_2=\emptyset$. So $T_3=(T^+_3,T^-_3)$ with
$$T^+_3:=(T^+_1\cup T^+_2)\setminus(T^-_1\cup T^-_2),~~ T^-_3:=(T^-_1\cup T^-_2)\setminus(T^+_1\cup T^+_2)$$ is a $[t]$-trade of volume $|T_1\oplus T_2|/2$ where $\oplus$ denotes symmetric difference.
Now let
  $$T_4:=\langle\{1\},\ldots,\{i\},\{t\},\{t+4\},\ldots,\{2t-i+3\}\rangle,$$
  with $T^+_4$ ($T^-_4$) being the set of vectors of $T_4$ with an even (odd) weight.
We have $T^+_3\cap T^+_4=\emptyset$.
To see this, let $B\in T^+_3\cap T^+_4$. As $T^+_3\subseteq T^+_1\cup T^+_2$, we have $B\in T^+_4\cap(T^+_1\cup T^+_2)$.
The blocks of both $T^+_4$, $T^+_2$ have even weights while those of $T^+_1$ have  odd weights. It follows that $B\in T^+_4\cap T^+_2\subseteq \langle\{1\},\ldots,\{i\}\rangle$. So $B\in\langle\{1\},\ldots,\{i\}\rangle$ with an even weight  and so $B\in T^-_1$ which implies that
$B\not\in T^+_3$, a contradiction. Analogously we have $T^-_3\cap T^-_4=\emptyset$.
So $T_5:=(T^+_5,T^-_5)$ with
\begin{equation}\label{T5}
T^+_5:=(T^+_3\cup T^+_4)\setminus(T^-_3\cup T^-_4),~~ T^-_5:=(T^-_3\cup T^-_4)\setminus(T^+_3\cup T^+_4)
\end{equation}
is a $[t]$-trade similarly. For its volume we have
\begin{align*}
2\vol(T_5)&=|T_1\oplus T_2\oplus T_4|\\
&=|T_1|+|T_2|+|T_4|-2|T_1\cap T_2|-2|T_1\cap T_4|-2|T_2\cap T_4|+4|T_1\cap T_2\cap T_4|\\
&=2^{t+1}+2^{t+1}+2^{t+1}-2\cdot2^{t-1}-2\cdot2^{i+1}-2\cdot2^i+4\cdot2^i\\
&=2(2^{t+1}+2^{t-1}-2^i),
\end{align*}
as required.

(ii) Let $T_j=(T^+_j,T^-_j)$ for $j=1,2,3$ be as in the case (i) and
  $$ T_4:=\langle\{1\},\ldots,\{i\},\{t\},\{t+1\},\{t+4\},\ldots,\{2t-i+2\}\rangle,$$
  with $T^+_4$ ($T^-_4$) to be the set vectors of $T_4$ of even (odd) weight.
Here we have $T^+_3\cap T^+_4=T^-_3\cap T^-_4=\emptyset$. We define $T_5:=(T^+_5,T^-_5)$ similar to \eqref{T5}. So it is a $[t]$-trade with
\begin{align*}
2\vol(T_5)&=|T_1\oplus T_2\oplus T_4|\\
&=|T_1|+|T_2|+|T_4|-2|T_1\cap T_2|-2|T_1\cap T_4|-2|T_2\cap T_4|+4|T_1\cap T_2\cap T_4|\\
&=2^{t+1}+2^{t+1}+2^{t+1}-2\cdot2^{t-1}-2\cdot2^{i+2}-2\cdot2^i+4\cdot2^i\\
&=2(2^{t+1}+2^{t-1}-3\cdot2^i),
\end{align*}
as desired.
}\end{proof}
From Corollary~\ref{c:form(i)}, Theorems~\ref{l:Kro2.5} and \ref{TradeContruaction}, we have the following.
\begin{cor} The spectrum of volumes of $[t]$-trades $T$ with $2\cdot2^t<\vol(T)<2.5\cdot2^t$ is
$$\{2^{t+1}+2^{t-1}-2^i : i=0,\ldots,t-2\}\cup\{2^{t+1}+2^{t-1}-3\cdot2^i : i=0,\ldots,t-3\}.$$
\end{cor}

\section{Characterization of small $[t]$-trades for $t=1,2$ }\label{s:vol3,6}

We say that two trades are \emph{equivalent}
if one is obtained from the other by some
permutation of the elements of $V$,
some  shifts, and, optionally, the swap of the two components $T_+$, $T_-$ of the trade.
In this section we characterize $[1]$-trades of volume $3$ and $[2]$-trades of volume $6$ up to equivalence.

\subsection[{[1]}-trades of volume 3]{[1]-trades of volume $3$} \label{s:vol3}
By the definition,
a small $[1]$-trade has volume
smaller than $4$.
Lemma~\ref{l:min} describes
the $[1]$-trades of minimum nonzero volume $2$;
the remaining value is considered in the following simple theorem.

\begin{thm}
\label{th:[1]3}
Every $[1]$-trade of volume $3$
is a shift of
$(\{Y_1,Y_2,Y_3\},\{Z_1,Z_2,Z_3\})$,
where
$Y_1$, $Y_2$, $Y_3$ are mutually disjoint,
$Z_1$, $Z_2$, $Z_3$ are also mutually disjoint,
$Y_1Y_2Y_3=Z_1Z_2Z_3$,
and $Y_i\ne Z_j$ for every $i,j\in\{1,2,3\}$.
\end{thm}
\begin{proof}{
Let $(T_+,T_-)$ be a $[1]$-trade,
then every element $i$ occurs
in the same number of blocks
from $T_+$ and from $T_-$.
If this number is $2$ or $3$,
then we consider the $x_i$-shift,
for which it is $1$ or $0$.
Making this for all elements,
we get a $[1]$-trade satisfying
the conditions
from the conclusion of the theorem.}
\end{proof}

\subsection[{[2]}-trades of volume 6]{$[2]$-trades of volume $6$}\label{s:vol6}
In the following four propositions, we define four types of $[2]$-trades of volume $6$.
The main result of this section states that every $[2]$-trade of volume $6$ is of one of these four types.

\begin{pro}\label{l:3-3}
Assume that a $[2]$-trade $T = (T_+,T_-)$ of volume $6$ is represented as
$$
T =
(1-XY_1)(1-XY_2)(1-XY_3)-
(1-XZ_1)(1-XZ_2)(1-XZ_3)
$$
where $X$, $Y_1$, $Y_2$, $Y_3$ are mutually disjoint sets,
$X$, $Z_1$, $Z_2$, $Z_3$ are also mutually disjoint sets,
$Y_1$, $Y_2$, $Y_3$, $Z_1$, $Z_2$, $Z_3$
are mutually different nonempty sets,
and $Y_1Y_2Y_3=Z_1Z_2Z_3$
(we note that $X$ can be empty
and a relation of type
$Z_i=Y_jY_k$ is possible).
Then, every extension $T'$ of
$T $
has the same form, up to a shift.
\end{pro}
\begin{proof}{
We have
$$T_+ = \{XZ_1,XZ_2,XZ_3,Y_1Y_2,Y_2Y_3,Y_1Y_3\},\quad
T_- = \{XY_1,XY_2,XY_3,Z_1Z_2,Z_2Z_3,Z_1Z_3\}.$$

By Lemma~\ref{l:sub} and the definition, an extension $(T'_+,T'_-)$ has the form $T'_+= S_+ \uplus x_s Q_+$,
$T'_-= S_- \uplus x_s Q_-$,
where
$T_+= S_+ \uplus Q_+$,
$T_-= S_- \uplus Q_-$, and
$S=(S_+,S_-)$, $Q=(Q_+,Q_-)$ are $[1]$-trades.
(Note that the multiset union $\uplus$ is essential here,
as some blocks can have multiplicity $2$;
e.g., if $XZ_1=Y_2Y_3$.)
W.l.o.g., we may assume that $\vol(Q)\le3$ (otherwise, we consider the $x_s$-shift).
If it is $0$, the statement holds trivially; $1$ is not possible by Lemma~\ref{l:min}. So it suffices to consider the following two cases.

\noindent{\bf Case 1.} $\vol(Q)=2$.

 It is not difficult to see that $Q$ cannot be a subtrade of $T$.
Indeed, if
$Q_+ = \{Y_1Y_2,Y_1Y_3\}$
(similarly, $\{Y_1Y_2,Y_2Y_3\}$ or $\{Y_1Y_3,Y_2Y_3\}$),
then every element of $Y_1$ occurs twice in the blocks of $Q_+$.
The same should be true for $Q_-$;
so, either $Q_-$ contains $XY_1$, or $Q_- =\{Z_iZ_j,Z_iZ_k\}$.
In the first case,
utilizing the definition of a $[1]$-trade,
we see that the second block of $Q_-$ is $XY_1Y_2Y_3$,
which is not a block from $T_-$, a contradiction.
In the second case, taking into account that $Y_1Y_2Y_3=Z_iZ_jZ_k$,
we conclude that $Y_1=Z_i$, which does not fit the hypothesis of the proposition.

If
$Q_+ = \{XZ_1,XZ_2\}$
(similarly, $\{XZ_1,XZ_3\}$ or $\{XZ_2,XZ_3\}$), then the elements of $Z_3$ do not occur in the blocks of $Q_+$.
The same should be true for $Q_-$.
So, $Q_-$ does not contain $Z_1Z_3$ or $Z_2Z_3$. If it contains $Z_1Z_2$, then the second block is $X$, which is not from $T_-$, again a contradiction. Therefore,
$Q_+ = \{XY_i,XY_j\}$ and w.l.o.g.,
$Q_+ = \{XY_1,XY_2\}$.
But this leads to $Z_1Z_2=Y_1Y_2$,
and from $Z_1Z_2Z_3=Y_1Y_2Y_3$ we find that  $Z_3=Y_3$, which contradicts the hypothesis of the proposition.

If
$Q_+ = \{XZ_1,Y_1Y_2\}$
(similarly, every remaining case),
then we can assume that
$Q_- = \{XY_i,Z_jZ_k\}$
(the other cases are shown above).
From $XZ_1Y_1Y_2 = XY_iZ_jZ_k$ we see that $Q_- = \{XY_3,Z_2Z_3\}$.
We now see that every element occurs exactly twice in blocks of $Q_+ \cup Q_-$.
By the definition of a $[1]$-trade,
every element occurs exactly once in blocks of $Q_+$ (similarly,  $Q_-$).
But this means that $Z_1=Y_3$, a contradiction.

\noindent{\bf Case 2.} $\vol(Q)=3$ (and so  $\vol(S)=3$).

Either $Q_+$, or $S_+$ contains $XZ_i$ and $XZ_j$ for some different $i$ and $j$.
W.l.o.g. we can assume that
$Q_+$ contains $XZ_1$, $XZ_2$. Consider the following two subcases.

(2a) {$Q_+ = \{XZ_1,XZ_2,XZ_3\}$}.
All elements of $Z_1Z_2Z_3$
occur exactly once
in the blocks of $Q_+$ and,
hence, in the blocks of $Q_-$.
So, $Q_-$ cannot have two blocks
from $Z_1Z_2$, $Z_1Z_3$, $Z_2Z_3$
and must have at least two blocks
from $XY_1$, $XY_2$, $XY_3$.
The third block of $Q_-$
is uniquely determined and
$Q_-=\{ XY_1,XY_2,XY_3\}$.
We see that the claim
of the proposition holds with
$$
T' =
(1-X'Y_1)(1-X'Y_2)(1-X'Y_3)-
(1-X'Z_1)(1-X'Z_2)(1-X'Z_3),
\qquad X'=x_sX.
$$

(2b) W.l.o.g., let {$Q_+ = \{XZ_1,XZ_2,Y_1Y_2\}$}.
We can assume that $Q_-$ contains
$XY_i, Z_1Z_j$ for some $i\in\{1,3\}$,
$j\in\{2,3\}$ (the other possibilities are similar or considered in the subcase (2a).
Then the third element of $Q_-$ is
$W=XZ_1 \oplus XZ_2 \oplus Y_1Y_2  \oplus XY_i  \oplus Z_1Z_j=XZ_2Z_jY_1Y_2Y_i$.

If $j=2$, then $W$ can only be $XY_2$, in which case
\begin{equation}\label{eq:j2}
T' =
(1-XY_1)(1-XY_2)(x_s-XY_3)-
(1-XZ_1)(1-XZ_2)(x_s-XZ_3).
\end{equation}
Then, the $x_s$-shift of
$T$ has the required form.

If $j=3$ and $i=3$, then $W=XZ_1$, which is not a block of $T_-$.

If $j=3$ and $i=1$, we have
$Q_-=\{ XY_1, Z_1Z_3, W\}$,
where
$W=XZ_2Z_3Y_2$
should be a block of $T_-$.
Clearly, $W\ne XY_2$
and $\ne Z_2Z_3$;
also $W\ne XY_1$
(as $Z_2Z_3\ne Y_1Y_2$
by the proposition hypothesis)
and, similarly, $W\ne XY_3$.
If $W=Z_1Z_2$, then $XY_2=Z_1Z_3$, which is possible, but then
$Q_-=\{XY_1,Z_1Z_3=XY_2, Z_1Z_2\}$
corresponds to (\ref{eq:j2}), considered above.
Finally, if $W=Z_1Z_3$, then we have $Z_1Z_2=XY_2$,
which means that
$XZ_3=Y_1Y_2$
and leads to the subcase (2a).
}
\end{proof}

\begin{pro}\label{l:2-2}
Assume that a $[2]$-trade $T = (T_+,T_-)$ of volume $6$ is represented as
$$
T =
(1-Y_1)(1-Y_2)(1-Y_3)-
(1-Z_1)(1-Z_2)(1-Z_3)
$$
where $Y_1$, $Y_2$, $Y_3$ are mutually disjoint nonempty sets, and likewise
$Z_1$, $Z_2$, $Z_3$ are mutually disjoint nonempty sets,
$Y_1$, $Y_2$, $Y_3$, $Z_1$, $Z_2$, $Z_3$
are mutually different nonempty sets,
and $Y_1Y_2=Z_1Z_2$.
Then every extension $T'$ of
$T $
has the same form, up to a shift.
\end{pro}
\begin{proof}{
We have
$$T_+ = \{Z_1,Z_2,   Y_2Y_3,Y_1Y_3,   Z_3,Z_1Z_2Z_3\},\quad
T_- = \{Y_1,Y_2,   Y_3,Y_1Y_2Y_3,   Z_2Z_3,Z_1Z_3\}.$$
 Repeating the arguments of the previous proof, we conclude that we have to check
 all possibilities for a $[1]$-subtrade $Q=(Q_+,Q_-)$ of volume $2$ or $3$.

 Denote
 $$\boldsymbol X:=\{Z_1,Z_2,\underline{Y_1},\underline{Y_2}\},\ \ \boldsymbol Y:=\{Y_2Y_3,Y_1Y_3,\underline{Y_3},\underline{Y_1Y_2Y_3}\},\ \ \boldsymbol Z:=\{Z_3, Z_1Z_2Z_3, \underline{Z_2Z_3}, \underline{Z_1Z_3}\}$$
 (the underlined blocks are from $T_-$, the other are from $T_+$).
 We first note the following fact.

(*) \em The sets $Q_+$ and $Q_-$ have the same number of elements from each of $\boldsymbol X$, $\boldsymbol Y$, $\boldsymbol Z$. \em

 Indeed, since $Y_3$ and $Z_3$ are different, we have $Y_3\backslash Z_3 \neq \emptyset$ or $Z_3\backslash Y_3 \neq \emptyset$.
 Assume w.l.o.g. that $Z_3\backslash Y_3$ is not empty; i.e., it contains some element $x_i$.
 By Lemma~\ref{l:sub}, $Q_+\cap \boldsymbol Z$ and $Q_-\cap \boldsymbol Z$ are the legs of a $[0]$-trade;
 hence, the cardinalities of this intersection are equal.
 Next, consider an element $x_j$ from $Y_3$. If $x_j\not\in Z_3$, then, similar to the argument above, we obtain that
 $|Q_+ \cap \boldsymbol Y| = |Q_- \cap \boldsymbol Y|$.
 If $x_j\in Z_3$, then we have $|Q_+ \cap( \boldsymbol Y \cup \boldsymbol Z)| = |Q_- \cap (\boldsymbol Y \cup \boldsymbol Z)|$.
 In any case, the whole statement of (*) follows.

\noindent{\bf Case 1.} $\vol(Q)=2$.

 Assume that
 $Q_+$ has one block from $\boldsymbol X$, say $X$,
 and one block from $\boldsymbol Y$, say $Y$.
 Then, from (*),
 $Q_+$ also has one block from $\boldsymbol X$,
 say $X'$, and one block from $\boldsymbol Y$, say $Y'$.
 We have $XX' = Z_iY_j$ and $YY'=Y_k$
 for some $i,j,k\in\{1,2\}$.
 In any case, $XX'YY'=Z_l$  for some $l\in\{1,2\}$,
 which contradicts Lemma~\ref{l:[1]}.
 So, $Q_+$ cannot have one block from $\boldsymbol X$ and one from $\boldsymbol Y$.
 Similarly, $Q_+$ cannot have one block from $\boldsymbol X$ and one from $\boldsymbol Z$,
 or one block from $\boldsymbol Y$ and one from $\boldsymbol Z$. The remaining
 possibilities satisfy the statement of the proposition:

 (a) {$Q_+ =  \{Z_1,Z_2\}$, $Q_- =  \{Y_1,Y_2\}$}; then the extension of $T$ is
 $$T' =x_s((1-Y_1)(1-Y_2)(1-x_sY_3) - (1-Z_1)(1-Z_2)(1-x_sZ_3)).$$

 (b) {$Q_+ =  \{Y_2Y_3,Y_1Y_3\}$, $Q_- =  \{Y_3,Y_1Y_2Y_3\}$}; then the extension of $T$ is
 $$T' =(1-Y_1)(1-Y_2)(1-x_sY_3) - (1-Z_1)(1-Z_2)(1-Z_3).$$

 (c) {$Q_+ =  \{Z_3,Z_1Z_2Z_3\}$, $Q_- =  \{Z_2Z_3,Z_1Z_3\}$}; then the extension of $T$ is
 $$T' =(1-Y_1)(1-Y_2)(1-Y_3) - (1-Z_1)(1-Z_2)(1-x_sZ_3).$$


\noindent{\bf Case 2.} $\vol(Q)=3$

$Q_+$ cannot intersect
one of $\boldsymbol X$, $\boldsymbol Y$, $\boldsymbol Z$
in two blocks,
otherwise it contains a $[1]$-subtrade of volume $2$ ((a), (b), or (c))
and the difference would be a $[1]$-trade of volume $1$.
So, $Q_+=\{X,Y,Z\}$ and $Q_-=\{X',Y',Z'\}$
for some $X$, $Y$, $Z$, $X'$, $Y'$, $Z'$
from
$\boldsymbol X \cap T_+ $,
$\boldsymbol Y \cap T_+ $,
$\boldsymbol Z \cap T_+ $,
$\boldsymbol X \cap T_- $,
$\boldsymbol Y \cap T_- $,
$\boldsymbol Z \cap T_- $,
respectively.
We have $YY' = Y_i$ and $ZZ'=Z_j$, where $i,j\in\{1,2\}$.
Assume w.l.o.g. that $YY' = Y_1$ and $ZZ'=Z_1$.
It follows from Lemma~\ref{l:[1]} that $XX'=Y_1Z_1$.
With these assumptions, $X'$, $Y'$,
and $Z'$ are uniquely determined by $X$, $Y$, and $Z$.
It remains to consider the eight possibilities to choose $X$, $Y$, and $Z$
($X\in \{Z_1,Z_2\}$, $Y \in \{Y_1 Y_3, Y_2 Y_3\} $, $Z\in \{Z_3, Z_1Z_2Z_3\}$).
The following two possibilities  are in agree with the proposition statement:

(d) {$Q_+ =  \{Z_1, Y_1 Y_3, Z_3  \}$}, $Q_- =  \{Y_1, Y_3, Z_1Z_3  \}$; then the extension of $T$ is
$$T' =(1-x_sY_1)(1-Y_2)(1-Y_3) - (1-x_sZ_1)(1-Z_2)(1-Z_3).$$

(e) {$Q_+ =  \{Z_2, Y_2 Y_3, Z_1Z_2Z_3 \}$}, $Q_- =  \{Y_2, Y_1 Y_2 Y_3, Z_2Z_3 \}$; then the extension of $T$ is
$$T' =x_s((1-x_sY_1)(1-Y_2)(1-Y_3) - (1-x_sZ_1)(1-Z_2)(1-Z_3)).$$

Consider the six other possibilities to choose $X$, $Y$, $Z$ from
$\boldsymbol X \cap T_+ $,
$\boldsymbol Y \cap T_+ $,
$\boldsymbol Z \cap T_+ $.
For example, let {$Q_+ =  \{Z_2, Y_2 Y_3, Z_3  \}$} (the other five cases are similar);
so, $Q_- =  \{Y_2, Y_1 Y_2 Y_3, Z_1Z_3  \}$.
Subtracting $(Q_-,Q_+)$ from the $[1]$-trade (e) above, we get
$$(\{Z_1Z_2Z_3,Z_1Z_3 \},\{Z_2Z_3,Z_3 \}),$$ which is not a $[1]$-trade (compare with (c) above).
Hence, $(Q_-,Q_+)$ is not a $[1]$-trade either.

Therefore, under the assumption that $YY' = Y_1$ and $ZZ'=Z_1$, in only two subcases, (d) and (e), we have trades.
The other cases ($YY' = Y_1$ and $ZZ'=Z_2$, $YY' = Y_2$ and $ZZ'=Z_1$, $YY' = Y_2$ and $ZZ'=Z_2$) are similar.}
\end{proof}

\begin{pro}\label{l:1-3}
Let a $[2]$-trade
$T = (T_+,T_-)$
of volume $6$ be represented as
$$
T =
(1-Y_1)(1-Y_2)(1-Y_3)-
(1-Z_1)(1-Z_2)(1-Y_1Y_2Y_3)
$$
where $Y_1$, $Y_2$, $Y_3$, $Z_1$, $Z_2$
are mutually disjoint nonempty sets.
Then, every extension $T'$ of
$T $
has the same form, up to a shift.
\end{pro}
\begin{proof}{
 We have
$$T_+ = \{Y_1Y_2,Y_1Y_3,Y_2Y_3, Z_1, Z_2, Y_1 Y_2 Y_3 Z_1Z_2\},\quad
T_- = \{Y_1,Y_2,Y_3, Z_1Z_2, Z_1Y_1Y_2Y_3, Z_2Y_1Y_2Y_3 \}.$$
 Repeating the arguments of the  proofs of Propositions \ref{l:3-3} and \ref{l:2-2}, we need to check
 all possibilities for a $[1]$-subtrade $Q=(Q_+,Q_-)$ of volume $2$ or $3$.

 Denote
 \begin{align*}
\boldsymbol Y&:=\{\underline{Y_1},\underline{Y_2},\underline{Y_3},Y_1 Y_2,Y_1 Y_3, Y_2 Y_3\},\\
\boldsymbol Z&:=\{Z_1,\underline{Z_1Z_2},\underline{Z_1Y_1 Y_2 Y_3}, Z_1Z_2Y_1 Y_2 Y_3\},\\
\boldsymbol Z'&:=\{Z_2,\underline{Z_1Z_2},\underline{Z_2Y_1 Y_2 Y_3}, Z_1Z_2Y_1 Y_2 Y_3\}.
 \end{align*}
 Similarly to the claim (*) in the proof of Proposition~\ref{l:2-2}, we have

 (*) \em $Q_+$ and $Q_-$ have the same number of elements from each of $\boldsymbol Z$ and $\boldsymbol Z'$. \em

Now, assume that $Q$ is a $[1]$-subtrade of volume $2$ or $3$. Consider the following four cases, which exhaust all possibilities.

\noindent{\bf Case 1.} $|Q_+ \cap \boldsymbol Z| = 2$ or $|Q_+ \cap \boldsymbol Z'| = 2$.

Without loss of generality assume $|Q_+ \cap \boldsymbol Z| = 2$.
Necessarily we have $|Q_- \cap \boldsymbol Z| = 2$, and so
 $Q_+ \supseteq \{Z_1,Z_1Z_2Y_1 Y_2 Y_3\}$, $Q_- \supseteq \{Z_1Z_2,Z_1Y_1 Y_2 Y_3\}$.
 We see that $(\{Z_1,Z_1Z_2Y_1 Y_2 Y_3\},\{Z_1Z_2,Z_1Y_1 Y_2 Y_3\})$ is a $[1]$-trade,
 and we cannot add one more element to each leg keeping the $[1]$-trade property.
 So, $\vol(Q)=2$ and
 $$
T' =
(1-Y_1)(1-Y_2)(1-Y_3)-
(1-x_sZ_1)(1-Z_2)(1-Y_1Y_2Y_3).
$$

\noindent{\bf Case 2.} $|Q_+ \cap \boldsymbol Z| =|Q_+ \cap \boldsymbol Z'| = 0$.

In this case we have  $Q_+ \subseteq \{Y_1Y_2,Y_1Y_3,Y_2Y_3\}$ and $Q_- \subseteq \{Y_1,Y_2,Y_3\}$.
The leg $Q_+$ has two intersecting blocks, but the blocks of $Q_-$ are mutually disjoint;
we have an obvious contradiction with the definition of a $[1]$-trade.

\noindent{\bf Case 3.}   $|Q_+ \cap \boldsymbol Z|= 1$ and $|Q_+ \cap \boldsymbol Z'| = 0$ (similarly, $|Q_+ \cap \boldsymbol Z|= 0$ and $|Q_+ \cap \boldsymbol Z'| = 1$).

From (*) we have that $Z_1\in Q_+$, $Z_1Y_1Y_2Y_3 \in Q_-$, and every other block of $Q_+$ or $Q_-$ belongs to $\boldsymbol Y$.
Since all elements of $Y_1Y_2Y_3$ occur in $Q_-$, at least two of $Y_1Y_2$, $Y_1Y_3$, $Y_2Y_3$ belong to $Q_+$ (in particular, the volume of $Q$ is $3$, not $2$).
W.l.o.g. assume $Q+=\{Z_1,Y_1Y_2,Y_1Y_3\}$. We see that the elements of $Y_1$ occurs twice in $Q_+$; hence, $Q_-$ contains $Y_1$.
By Lemma~\ref{l:[1]}, the third block in $Q_-$ is $Z_1\oplus Y_1Y_2\oplus Y_1Y_3 \oplus Z_1Y_1Y_2Y_3  \oplus Y_1$, i.e., $\emptyset$.
Since $\emptyset\not\in T_-$, hence we reach at a contradiction.

\noindent{\bf Case 4.}  $|Q_+ \cap \boldsymbol Z| =|Q_+ \cap \boldsymbol Z'| = 1$.

 Consider the following subcases.

(4a) {$Z_1Z_2Y_1Y_2Y_3 \in Q_+$, $Z_1Z_2 \in Q_-$, the other blocks are from $\boldsymbol Y$}. Since all elements of $Y_1Y_2Y_3$ occur in $Q_+$,
$Q_-$ must contain each of $Y_1$, $Y_2$, $Y_3$, which is impossible since $|Q_-|\le 3$.

(4b) {$Z_1,Z_2 \in Q_+$, $Z_1Y_1Y_2Y_3,Z_2Y_1Y_2Y_3 \in Q_-$, the other blocks are from $\boldsymbol Y$}.
Since all elements of $Y_1Y_2Y_3$ occur in $Q_-$ twice,
$Q_+$ must contain each of $Y_1Y_2$, $Y_1Y_3$, $Y_2Y_3$, which is impossible as $|Q_+|\le 3$.

(4c) {$Z_1Z_2Y_1Y_2Y_3 \in Q_+$, $Z_1Y_1Y_2Y_3,Z_2Y_1Y_2Y_3 \in Q_-$, the other blocks are from $\boldsymbol Y$}.
Since $Z_1Z_2Y_1Y_2Y_3 \oplus Z_1Y_1Y_2Y_3 \oplus Z_2Y_1Y_2Y_3 = Y_1Y_2Y_3 \not\in T_+$,
 from Lemma~\ref{l:[1]} we observe that the $[1]$-trade $(Q_+,Q_-)$ cannot have volume $2$.
So, $Q_+$ has two elements from $\boldsymbol Y$, say $Y_iY_j$ and $Y_iY_k$.
By Lemma~\ref{l:[1]} we find $Y_i\in Q_-$, and so
$$
T' =
(1-x_sY_i)(1-Y_j)(1-Y_k)-
(1-Z_1)(1-Z_2)(1-x_sY_1Y_2Y_3).
$$

(4d) {$Z_1,Z_2 \in Q_+$, $Z_1Z_2 \in Q_-$, the other blocks are from $\boldsymbol Y$}.
Similarly to the  subcase (4c), we have
$$
T' =
x_s(1-x_sY_i)(1-Y_j)(1-Y_k)-
x_s(1-Z_1)(1-Z_2)(1-x_sY_1Y_2Y_3).
$$
}\end{proof}

\begin{pro}\label{l:1-1}
Assume that
$$T=(\{Y_1,Y_2,Y_3,XZ_1,XZ_2,XZ_3\},
\{Z_1,Z_2,Z_3,XY_1,XY_2,XY_3\}),$$
where
$X$, $Y_1$, $Y_2$, $Y_3$
are mutually disjoint,
$X$, $Z_1$, $Z_2$, $Z_3$
are mutually disjoint,
$Y_1Y_2Y_3=Z_1Z_2Z_3$,
 $Y_i\ne Z_j$
for every $i,j\in\{1,2,3\}$ and $X\ne\emptyset$
Then, every extension of $T$
has the same form,
up to a shift.
\end{pro}

Proposition~\ref{l:1-1} is a partial case of the following more general fact.

\begin{pro}\label{l:1-1gen}
Assume that
$$T=(1-X)\bar\sigma$$
where $\bar\sigma$ is a $[t-1]$-trade of volume less than $2^t$ (i.e., small)
and $X$ is a nonempty set, disjoint from the blocks of $\bar\sigma$
(so, $T$ is a small $[t]$-trade). 
Let $T'$ be an extension of $T$.
Then
\begin{equation}\label{eq:dwq}
T'=(1-x_sX)\bar\sigma, \qquad T'=(x_s-X)\bar\sigma,
\end{equation}
or
\begin{equation}\label{eq:xvj}
T'=(1-X)\bar\sigma',
\end{equation}
where $\bar\sigma'$ is an extension of $\bar\sigma$.
\end{pro}
\begin{proof}{
We have $T'=x_s\bar\kappa+(T-\bar\kappa)$, where $\bar \kappa$ is a $[t-1]$-subtrade of $T$.
W.l.o.g., we can assume that $\bar\kappa$ is small.
Let $\bar\kappa^p$ be the projection of $\bar\kappa$ in $X$.
Then $\bar\kappa^p$ is a small $[t-1]$-trade, whose blocks are blocks of $\bar\sigma$.
Let us prove the following claim:

(*) \emph{If $\bar\kappa^p$ is not void, then all blocks of the $[t-1]$-trade $\bar\kappa^p+\bar\sigma$ have even multiplicity.}

Denote by $a$ and $b$ the number of different blocks of $\bar\sigma$ of odd and even multiplicity, respectively.
The volume of $\bar\sigma$ is at least $(a+2b)/2$; since $\bar\sigma$ is a small $[t-1]$-trade, we have
\begin{equation}\label{eq:a2b}
(a+2b)/2 < 2^t.
\end{equation}
Denote by $a'$ and $b'$ the number of blocks of $\bar\kappa^p$ of odd multiplicity whose multiplicity in $\bar\sigma$ is odd and even, respectively.
So, the number of blocks of odd multiplicity in $\bar\kappa^p+\bar\sigma$ is $a-a'+b'$.

Next, since $\bar\kappa^p$ is a small non-void $[t-1]$-trade, by Lemma~\ref{l:odd} we have
\begin{equation}\label{eq:a2t}
a'+b' \ge 2^t
\end{equation}

Now, using (\ref{eq:a2b}), (\ref{eq:a2t}),
and the trivial fact that $b'\le b$, for the number $a-a'+b'$ of odd-multiplicity blocks of $\bar\kappa^p+\bar\sigma$ we have
$$ a-a'+b' = a+2b' -a'-b' \le  (a+2b) -(a'+b') < 2\cdot 2^t - 2^t=2^t.$$
By Lemma~\ref{l:odd}, this number is $0$. Hence (*) follows.

If $\bar\kappa^p$ is void, we have (\ref{eq:xvj}). By (*), it remains to consider the case
when all blocks of the $[t-1]$-trade $\bar\kappa^p+\bar\sigma$ have even multiplicity.

(**) \emph{$(\bar\kappa^p+\bar\sigma)/2$ is a $[t-1]$-subtrade of $\bar\sigma$.} Equivalently,
every block of $(\bar\kappa^p+\bar\sigma)/2$ has the same sign in $(\bar\kappa^p+\bar\sigma)/2$ as in $\bar\sigma$
and at most the same multiplicity.
Indeed, by the definition of $\bar\kappa^p$,
the coefficient $\alpha$ at each of its blocks
satisfies $|\alpha|\le|\beta|$. It follows that $0\le \left|\frac{\alpha+\beta}2 \right| \le |\beta|$ and  $\frac{\alpha+\beta}2$ and $\beta$ are of the  same sign. So (**) follows.

Since $\bar\sigma$ is a small $[t-1]$-trade, it does not have proper subtrades, and $(\bar\kappa^p+\bar\sigma)/2$ is either zero or $\bar\sigma$.
In the first case, $\bar\kappa^p=-\bar\sigma$, and  $\bar\kappa=-X\bar\sigma$. In the second case, $\bar\kappa^p=\bar\kappa=\bar\sigma$.
Therefore, in every case, we obtain that $T'$ has on the forms given  in (\ref{eq:dwq}).}
\end{proof}

\begin{thm}\label{t:[2]}
 Every $[2]$-trade of volume $5$ or $6$ have one of the forms described in Propositions~$\ref{l:3-3}$--$\ref{l:1-1}$.
\end{thm}
In particular, Theorem~\ref{t:[2]} implies that there are no $[2]$-trades of volume $5$, which is a known fact \cite{Hwang:86}.
\begin{proof}{
 We proceed by induction on the number of the elements involved in the blocks of a trade.
 If this number is zero, then the statement is trivial (there are no non-void trades),
 which gives the induction base.
 Let us consider a $[2]$-trade $T$ of volume $5$ or $6$.
 If it has a projection of volume $5$ or $6$, then by the inductive hypothesis the statement of the theorem holds for this projection.
 Hence, it is true for $T$, by Propositions~\ref{l:3-3}--\ref{l:1-1}.

 If $T$ has a void projection, then it has the form $T = (1-x_i)\bar\sigma$, where $\bar\sigma$ is a $[1]$-trade of volume $3$.
 In this case, the statement is straightforward from Theorem~\ref{th:[1]3}.

 It remains to consider the case when all projections have volume $4$.
 For a given $i$, the $i$-projection has the form
 $$(1-X)(1-Y)(1-Z) = 1-X-Y-Z+XY+XZ+YZ-XYZ,$$
 up to a shift. Then
 $$
 T = \alpha_{000}-\alpha_{100}X-\alpha_{010}Y-\alpha_{001}Z+\alpha_{110}XY+\alpha_{101}XZ+\alpha_{011}YZ-\alpha_{111}XYZ \pm (1-x_i) V \pm (1-x_i) W,
 $$
  where $\alpha_{000}, \alpha_{100}, \alpha_{010}, \alpha_{001}, \alpha_{110}, \alpha_{101}, \alpha_{011}, \alpha_{111} \in \{1,x_i\}$ and $V$, $W$ are some blocks with $i\not\in V,W$.
 The number of blocks of $T$ with (or without) element $i$ is at least $2$ and at most $10$;
 taking into account Lemma~\ref{l:sub}, it is $4$, $6$, or $8$.
 So, the number  $p_i$ of coefficients  $\alpha_{\cdots}$ equal to $x_i$ is $2$, $4$, or $6$. W.l.o.g. (up to the $x_i$-shift) we may assume that it is $p_i=2$ or $4$.
 The case of $p_i=2$, up to a shift and renaming  $X$, $Y$, and $Z$, is exhausted by the Cases 1-3 below.

 \noindent{\bf Case 1.}  $\alpha_{000}=\alpha_{100}=x_i$, the other coefficients are $1$:
 $$T = x_i-x_iX-Y-Z+XY+XZ+YZ-XYZ + (1-x_i) V - (1-x_i) W.$$
 Considering the $[1]$-subtrade $x_i-x_iX-x_iV+x_iW$, we see that $X$ and $V$ are disjoint and $W=XV$.
 We find that the case falls under the conditions of Proposition~\ref{l:1-1}, with $Y_1:=x_i$, $Y_2:=V$, $Y_3:=YZ$, $Z_1:=Y$, $Z_2:=Z$, $Z_3:=x_iXV$, and $X=X$.

\noindent{\bf Case 2.}  $\alpha_{000}=\alpha_{110}=x_i$, the other coefficients are $1$:
 $$T = x_i-X-Y-Z+x_iXY+XZ+YZ-XYZ + (1-x_i) V + (1-x_i) W.$$
 Considering the $[1]$-subtrade $x_i+x_iXY-x_iV-x_iW$, we see that $V$ and $W$ are disjoint and $VW=XY$.
 We find that the case falls under the conditions of Proposition~\ref{l:2-2}, with $Y_1:=X$, $Y_2:=Y$, $Y_3:=Z$, $Z_1:=V$, $Z_2:=W$, $Z_3:=x_i$.

\noindent{\bf Case 3.}  $\alpha_{000}=\alpha_{111}=x_i$, the other coefficients are $1$:
  $$T = x_i-X-Y-Z+XY+XZ+YZ-x_iXYZ + (1-x_i) V - (1-x_i) W.$$
 Similar to Case 1,  $XYZ$ and $V$ are disjoint and $W=XYZ\oplus V$.
 The case falls under the conditions of Proposition~\ref{l:1-3}, with $Y_1:=X$, $Y_2:=Y$, $Y_3:=Z$, $Z_1:=V$, $Z_2:=x_i$.

\noindent{\bf Case 4.}  $p_i=4$.

 We can assume that $p_j=4$ for any element $j$ involved in the trade $T$ (otherwise we will be in one of the Cases 1-3);
 so,

 (*) \emph{for every essential element $j$, in the decomposition $T = P+x_jP'$ the volume of the $[1]$-trades $P$ and $P'$ is $3$.}

 In particular,

 (**) \emph{$V$ and $W$ consist of elements of $XYZ$}, as any other element contradicts (*).

 (***) \emph{$VW = XYZ$} (indeed, from (*) we see that every element $j$ from $XYZ$ belongs to exactly one of $V$, $W$).

We consider two subcases.

(4a) Firstly, assume that one of $X$, $Y$, $Z$,  say $X$, has two different elements $j$ and $k$.
Since the $j$-projection of $T$ has volume $4$, we find that $V$ (and hence $W$) differs from
one of $1$, $X$, $Y$, $Z$, $XY$, $XZ$, $YZ$, $XYZ$ in only one element $j$.
Up to a shift, we assume that $V=x_j$.
The same can be said about $k$; so, $X=x_jx_k$. Now, neither $Y$ nor $Z$ can have more than one element (otherwise, there are projections of volume $6$).

Let, w.l.o.g., $\al_{000}=x_i$.
The  $[1]$-subtrade of $T$ consisting of all blocks containing $x_i$  has six blocks, three of which we know: $x_i$, $x_iV$, and $x_iW$.
The other three blocks must sum up to  $x_i \oplus x_iV \oplus x_iW=x_iXYZ$;
so, they are either $x_iX$, $x_iY$, $x_iZ$, or $x_iXY$, $x_iYZ$, $x_iXZ$.
The last case is impossible because the four blocks $x_i$, $x_iXY$, $x_iYZ$, $x_iXZ$ have the same sign.
We conclude that
 $$
 T = x_i-x_iX-x_iY-x_iZ+XY+XZ+YZ-XYZ -V+x_iV - W+x_i W,
 $$
where $X=x_jx_k$, $V=x_j$, $W=x_kYZ$,
which has the $j$-projection of volume $6$, contradicting our assumption.

(4b) The remaining subcase is $|X|=|Y|=|Z|=1$. Each of $V$, $W$ is one of $1$, $X$, $Y$, $Z$, $XY$, $XZ$, $YZ$, $XYZ$.
It is not difficult to conclude 
that, up to a shift,
$$T = 2-X-Y-Z-x_i +XYZ+XYx_i+XZx_i+YZx_i-2XYZx_i,$$
which is the case of Proposition~\ref{l:3-3}.}
\end{proof}

\section{Computational results} \label{s:comp}
In this section we present an algorithm to construct $[t]$-trades  with a given foundation of size $v$. We implement this algorithm and enumerate all
small $[t]$-trades for $t\le4$.

\subsection{Algorithm}\label{s:alg}
Corollary~\ref{c:decomp} allows to compute all possible $s$-small $[t]$-trades $T$ with a foundation of size $v$
if we know all $s$-small $[t]$-trades $T'$ and $s$-small $[t-1]$-trades $T''$ with foundations of size $v-1$.
This gives the possibility to classify, for a given $s$, all $s$-small $[t]$-trades of small foundation recursively, starting from $t=0$.
The following algorithm describes the recursive step.
\begin{enumerate}
 \item[0] Set $\mathcal T := \emptyset$.
 \item[1] For all $s$-small $[t]$-trades $T'$ and all $s$-small $[t-1]$-trades $T''$ do Steps 1.1--1.2.
 \begin{enumerate}
 \item[1.1] Add $ T' - (1-x_v)T'' $ to $\mathcal T$.
 \item[1.2] If $T'-T''$ is not small, then add $ x_v T' + (1-x_v)T'' $ to $\mathcal T$.
\end{enumerate}
\end{enumerate}
At the end, $\mathcal T$ will be the set of all $s$-small $[t]$-trades. Indeed, for every such trade $T$, consider the representation
$T = P + x_v P'$, where $v \not\in\found(P),\found(P')$.
If $P'$ is $s$-small, then $T$ is added at Step 1.1 with $T'=P+P'$ and $T''=P'$.
If $P'$ is not $s$-small, then $P$ is $s$-small, and $T$ is added at Step 1.2 with $T'=P+P'$ and $T''=P$.

From $\mathcal T$, we can choose a complete collection of nonequivalent $s$-small $[t]$-trades
(to be exact, representatives of all equivalence classes).
The graph isomorphism  routine  \cite{nauty2014} is employed to deal with
the equivalence rejection. 
See \cite{KO:alg} for general technique of representing subsets of $2^V$ by graphs, for checking the equivalence.
If we do not need the list of all trades, we can check equivalence at Steps 1.1 and 1.2, and collect only nonequivalent representatives.
In this case, there is an obvious improvement: it is sufficient to consider either only nonequivalent $[t]$-trades $T'$, or only nonequivalent $[t-1]$-trades $T''$.
However, the second component, $T''$ or $T'$, must be chosen from all different trades with corresponding parameters,
and this approach does not allow to make all steps of the recursion by considering only nonequivalent representatives.

\subsection{Validity of computational results}\label{s:validity}
The correctness of the computer classification can be partially verified by the following double-counting arguments
(see \cite[Ch.~10]{KO:alg}). Denote by $\aut(T)$ the full automorphism group of a trade $T$, which consists of all equivalence
transformations that send $T$ to itself (recall that an equivalence transformation consists of a shift, a permutation of the elements of $V$, and,
optionally, the swap of the components $T_+$, $T_-$ of the trade $T=(T_+,T_-)$).
The number of all different $s$-small $[t]$-trades with foundation contained in $V$ can be calculated as
\begin{equation}\label{eq:orbit-stabilizer}
\sum |\aut(2^V)|/|\aut(T)|,
\end{equation}
where the summation is over all nonequivalent representatives and $\aut(2^V)$ is the group of all equivalence transformations with
 $|\aut(2^V)|=2\cdot 2^v\cdot v!$.
On the other hand, this number can be found as the total number of solutions found by the algorithm
(if $T'$ or $T''$ runs over nonequivalent representatives,
then every solution is counted with the factor $2^v(v-1)!/|\aut(T')|$ or $2^v(v-1)!/|\aut(T'')|$, respectively).
Coinciding this number with (\ref{eq:orbit-stabilizer}) means that the probability of errors of different kinds is very-very small.

\subsection{Results: Construction of small $[t]$-trades with $t\le4$ and $|\found|\le7$ }\label{s:res}
The tables below show the number of $[t]$-trades in $2^V$,
for given $|V|$ and given volume.
The first number in a cell indicates the number of equivalence classes of all  $[t]$-trades.
The second number (in parentheses) indicates the number of equivalence classes
of non-degenerate  $[t]$-trades.
The third, the number of equivalence classes of all simple $[t]$-trades.
The fourth, the number of equivalence classes of non-degenerate simple $[t]$-trades.
Note that the row ``$v=...$'' reflects the numbers for trades in $2^V$ with $|V|=v$,
but the foundation size of the trades can be smaller;
so, the same trades are necessarily counted in the next row,
together with the trades of foundation size $v+1$.

$t=1$:
$$
\begin{array}{r||c|c|c|c|c|c|c|c|c|c|c|c|c|c}
 \mbox{vol.} & 0 & 2            & 3              \\ \hline\hline
 v\le 1      & 1 & 0            & 0              \\ \hline
 v=2         & 1 & 1(1){\ }1(1)   & 0              \\ \hline
 v=3         & 1 & 2(1){\ }2(1)   & 1(1){\ }0(0)     \\ \hline
 v=4         & 1 & 4(1){\ }4(1)   & 5(4){\ }3(3)     \\ \hline
 v=5         & 1 & 6(1){\ }6(1)   & 17(8){\ }13(7)   \\ \hline
 v=6         & 1 & 9(1){\ }9(1)   & 51(12){\ }44(11) \\ \hline
 v=7         & 1 & 12(1){\ }12(1) & 126(14){\ }115(13)            \\ \hline
\end{array}
$$

$t=2$:
$$
\begin{array}{r||c|c|c|c|}
 \mbox{vol.} & 0 & 4            & 6               & 7                           \\ \hline\hline
 v\le 2& 1 & 0             & 0                & 0                            \\ \hline
 v=3   & 1 & 1(1){\ }1(1)    & 0                & 0                    \\ \hline
 v=4   & 1 & 2(1){\ }2(1)    & 2(2){\ }0(0)       & 0                   \\ \hline
 v=5   & 1 & 4(1){\ }4(1)    & 12(9){\ }7(7)      & 7(7){\ }0(0)           \\ \hline
 v=6   & 1 & 7(1){\ }7(1)    & 43(17){\ }32(15)   & 88(63){\ }52(52)               \\ \hline
 v=7   & 1 &  11(1){\,}11(1)  &   130(24){\,}109(22) & 515(161){\,}391(148) \\ \hline
\end{array}
$$

$t=3$:
$$
\begin{array}{r||c|c|c|c|c|c|c|c|c|c|c|c|c|c}
 \mbox{vol.} & 0 & 8            & 12              & 14                & 15           \\ \hline\hline
 v\le 3      & 1 & 0            & 0               & 0                 & 0            \\ \hline
 v=4         & 1 & 1(1){\ }1(1)   & 0               & 0                 & 0            \\ \hline
 v=5         & 1 & 2(1){\ }2(1)   & 2(2){\ }0(0)      & 0                 & 1(1){\ }0(0)   \\ \hline
 v=6         & 1 & 4(1){\ }4(1)   & 15(11){\ }9(9)    & 14(14){\ }0(0)      & 7(6){\ }0(0)   \\ \hline
 v=7         & 1 & 7(1){\ }7(1)   & 56(20){\ }41(18)  & 165(110){\ }89(89)  & 74(51){\ }0(0) \\ \hline
\end{array}
$$

$t=4$:
$$
\begin{array}{r||c|c|c|c|c|c|c|c|c|c|c|c|c|c}
 \mbox{vol.} & 0 & 16           & 24              & 28                & 30           & 31 \\ \hline\hline
 v\le 4      & 1 & 0            & 0               & 0                 & 0            & 0 \\ \hline
 v=5         & 1 & 1(1){\ }1(1)   & 0               & 0                 & 0            & 0 \\ \hline
 v=6         & 1 & 2(1){\ }2(1)   & 2(2){\ }0(0)      & 0                 & 2(2){\ }0(0)   & 0 \\ \hline
 v=7         & 1 & 4(1){\ }4(1)   & 15(11){\ }9(9)    & 17(17){\ }0(0)      & 15(12){\ }0(0) & 0 \\ \hline
\end{array}
$$


\subsection{Proof of Lemma~\ref{l:computational}}

For $t=2$, we can further implement our algorithm to construct all $[t]$-trades $T$ with $2\cdot2^t\le\vol(T)\le3\cdot2^t$ and $|\found(T)|=5$.
In particular, Lemma~\ref{l:computational} is derived. The enumeration of these trades is given in the table below.

$$
\begin{array}{r||c|c|c|c|c|c|c|c|c|}
 \mbox{vol.}            & 8         & 9         & 10         & 11         & 12          \\ \hline\hline
 v\le 2                 & 0          & 0         & 0                & 0              & 0           \\ \hline
 v=3                    & 1(1){\ }0(0) & 0         & 0                & 0              & 1(1){\ }0(0)   \\ \hline
 v=4                  & 7(6){\ }2(2) & 2(2){\ }0(0) & 3(3){\ }0(0)       & 0              & 18(17){\ }0(0) \\ \hline
 v=5           &  94(80){\ }39(36) &  85(82){\,}0(0) &  479(471){\,}20(20) &   771(771){\,}0(0) &   3195(3154){\,}26(26)   \\ \hline
\end{array}
$$

\section*{Acknowledgments}
The research of the first author was in part supported by a grant from IPM (No. 96050211).
The research of the second and third authors was partially supported by  IPM.
The research of the fourth author was supported by
the Program of fundamental scientific researches of the SB RAS No.I.5.1. (project No. 0314-2019-0016).

\bibliographystyle{plain}
\bibliography{k.bib}

 \end{document}